\documentclass[english,10pt]{article}

\usepackage[english]{babel}
\usepackage[utf8x]{inputenc}
\usepackage[T1]{fontenc}
\usepackage[formats]{listings}
\usepackage{geometry}
\usepackage{enumitem}

\usepackage{graphicx}
\usepackage[colorinlistoftodos]{todonotes}
\usepackage{setspace}
\usepackage{placeins}
\usepackage{enumitem}
\usepackage{titlesec}
\usepackage{cite}
\usepackage{comment}
\usepackage{caption}
\usepackage{subcaption}
\usepackage{enumitem}

\usepackage[colorlinks=true, allcolors=blue]{hyperref}

\usepackage{amsmath}
\usepackage{amsthm}
\usepackage{amssymb}
\usepackage{amsfonts}
\usepackage{bbm}
\usepackage{bm}
\usepackage{nicefrac}
\usepackage{mathtools}

\usepackage{tikz}

\usepackage{graphicx}
\usepackage{fullpage}
\usepackage{float}
\usepackage{wrapfig}
\usepackage{caption}

\usepackage{algpseudocode}
\usepackage{algorithm}
\usepackage{listings}

\usepackage{amsthm}
\usepackage{dsfont}
\usepackage{array}
\usepackage{mathrsfs}
\usepackage{cite}
\usepackage{comment}
\usepackage{mathrsfs}
\usepackage{bm}

\makeatother

\usepackage{babel}
\usepackage{color}
\usepackage{centernot}

\usepackage{babel}
\usepackage{sansmath}

\newcommand{\bbeta}{\bm{\beta}}

\newcommand{\bgamma}{\boldsymbol{\gamma}} 
\newcommand{\bSigma}{\boldsymbol{\Sigma}}
\newcommand{\bsigma}{\boldsymbol{\sigma}} 

\newcommand{\bpsi}{\boldsymbol{\psi}}

\newcommand{\dif}{\mathrm{d}}

\newcommand{\R}{\mathbb{R}}

\title{On Naive Mean-Field Approximation for high-dimensional canonical GLMs}

\author{Sumit Mukherjee \thanks{Department of Statistics, Columbia University, New York, NY 10027, U.S.A. email: sm3949@columbia.edu.} \and Jiaze Qiu\thanks{Department of Statistics, Harvard University, Cambridge, MA 02138, U.S.A. email: jiazedengqiu@gmail.com. } \and 
  Subhabrata Sen\thanks{Department of Statistics, Harvard
    University, Cambridge, MA 02138, U.S.A. email: subhabratasen@fas.harvard.edu.} }

\theoremstyle{plain}\newtheorem{lemma}{\textbf{Lemma}}\newtheorem{theorem}{\textbf{Theorem}}\newtheorem{corollary}{\textbf{Corollary}}\newtheorem{assumption}{\textbf{Assumption}}\newtheorem{example}{\textbf{Example}}\newtheorem{definition}{\textbf{Definition}}

\theoremstyle{definition}

\theoremstyle{definition}

\makeatletter
\renewenvironment{proof}[1][\proofname] {
	\par\pushQED{\qed}\normalfont
	\topsep6\p@\@plus6\p@\relax
	\trivlist\item[\hskip\labelsep\bfseries#1\@addpunct{:}]
 	\ignorespaces
} {
	\popQED\endtrivlist\@endpefalse
}
\makeatother

\begin{document}

\maketitle

\begin{abstract}
We study the validity of the Naive Mean Field (NMF) approximation for canonical GLMs with product priors. This setting is challenging due to the non-conjugacy of the likelihood and the prior. Using the theory of non-linear large deviations \cite{MR3519474,MR3881829,austin2019structure}, we derive sufficient conditions for the tightness of the NMF approximation to the log-normalizing constant of the posterior distribution. As a second contribution, we establish that under minor conditions on the design, any NMF optimizer is a product distribution where each component is a quadratic tilt of the prior. In turn, this suggests novel iterative algorithms for fitting the NMF optimizer to the target posterior. Finally, we establish that if the NMF optimization problem has a ``well-separated maximizer", then this optimizer governs the probabilistic properties of the posterior. Specifically, we derive credible intervals with average coverage guarantees, and characterize the prediction performance on an out-of-sample datapoint in terms of this dominant optimizer. 
\end{abstract}





\section{Introduction}
High-dimensional data is ubiquitous in modern statistics and machine learning. Bayesian methods provide a powerful mechanism to extract the latent signal from the data. Under a Bayesian paradigm, the scientist constructs the posterior distribution, and draws statistical conclusions based on this probability distribution. In modern applications, the posterior distribution is typically high-dimensional. This motivates the development of basic tools to understand the properties of high-dimensional posterior distributions. 

The present article investigates this broader question in the concrete setting of generalized linear models (GLMs) \cite{mccullagh2019glms}. GLMs are canonical models for supervised learning. Given data  $\{(y_i, \mathbf{x}_i): 1\leq i \leq n\}$, GLMs, e.g., logistic/probit/poisson regression, are ubiquitous if the responses $y_i$ are either categorical or count data. These models generalize the well-known linear regression model. Throughout, we assume that one observes $\{(y_i, \mathbf{x}_i): 1\leq i \leq n\}$, $\mathbf{x}_i \in \mathbb{R}^p$; the matrix  $\mathbf{X}^{\top} = [\mathbf{x}_1 , \ldots, \mathbf{x}_n] \in \mathbb{R}^{p \times n}$ is referred to as the design or the measurement matrix. The responses $y_i$ are assumed to take values in a closed subset $\mathcal{X}$ of $\mathbb{R}$ e.g., $\mathcal{X} = \mathbb{R}$ for the linear model, $\mathcal{X} = \{0,1\}$ for logistic models and $\mathcal{X} = \mathbb{N}$ for the poisson regression model. We assume that the conditional distribution of $\mathbf{y} = (y_1, \cdots, y_n)$ is given by a canonical GLM, introduced below:

 The random variables $\{y_i\}_{1\le i\le n}$ are mutually independent, and the marginal distribution of $y_i$ is $P_i$, where $P_i$ comes from a natural exponential family (NEF) on $\mathcal{X}$, defined via the Radon-Nikodyn derivative $p_i:=\frac{dP_i}{d\lambda}$, where $p_i(\cdot)$ is given by
\begin{equation}
\label{eq:NEF}
    p_i(y_i) = \exp \left ( y_i \theta_i -b(\theta_i)  \right ).
\end{equation}
Here $\lambda$ is a probability measure on $\mathcal{X}$, $\theta_i = \mathbf{x}_i^{\top} \bbeta$ for $i\in [n]$, $\bbeta \in \mathbb{R}^p$ is the vector of regression coefficients. 
The function $b(\cdot)$ is the log normalizing constant of the p.d.f.~$p_i(\cdot)$, given by
\begin{align}\label{eq:b}
b(\theta_i):= \log\Big[\int_{\mathcal{X}}  \exp \left ( y_i \theta_i  \right )d\lambda(y_i)\Big].
\end{align}
We assume that the base measure $\lambda$ is such that $b(\theta_i)$ is finite for all $\theta_i\in \mathbb{R}$.
Standard exponential family theory (see for example \cite{wainwright2008graphical}) implies that $b(\cdot)$ is a smooth function, with $b(0)=0$. Before proceeding further, we provide two concrete examples of canonical GLMs. 

\begin{example} [Linear Regression] 
\label{def:linear_regression}
For linear regression, $\mathcal{X}=\mathbb{R}$ and $\lambda$ is the $\mathcal{N}(0,1)$ distribution. With $\theta_i=\mathbf{x}_i^T\bbeta$,  $p_i(y_i)$ corresponds to the $\mathcal{N}( \theta_i, 1)$ pdf.
In this case we have $b(\theta)=\frac{\theta^2}{2}.$
\end{example}
\begin{example} [Binary Logistic Regression and Binomial Logistic Regression] 
\label{def:logistic_regression}
For (binary) logistic regression, $\mathcal{X} = \{0,1\}$, and $\lambda$ is the $Bern(1/2)$ distribution. With $\theta_i=\mathbf{x}_i^T\bbeta$ as before, the model posits
\begin{align}\label{eq:logistic_regression_def}
p_i(y_i|\mathbf{x}_i)=\frac{2\exp(\theta_i y_i)}{1+\exp(\theta_i)},\quad y_i\in \{0,1\}.
\end{align} In this case we have
    $$    b(\theta) = \log (1 + \exp(\theta))-\log 2.$$
More generally, for binomial logistic regression, one takes $\mathcal{X} = \{ 0,1,\dots, N \} := [N]$ where $N$ is a known positive integer and $\lambda$ is the $Bin(N,1/2)$ distribution on $[N]$. In this case, we have,
$$p_i(y_i|\mathbf{x}_i)=\frac{2^N\exp(\theta_i y_i)}{(1+\exp(\theta_i))^N},\quad y_i\in\{ 0,1,\cdots,N\}.$$ 
In addition,
    $$    b(\theta) = N \log (1 + \exp(\theta)) - N \log 2.$$
\end{example}

\noindent 
The canonical GLM likelihood in \eqref{eq:NEF} can be written as 
\begin{equation}
\label{eq:likelihood} 
    \begin{aligned}
        L(\boldsymbol{\beta}) &= \prod_{i = 1}^{n} \exp \left (  y_i \theta_i -b(\theta_i)     \right ),
    \end{aligned}
\end{equation}
where $\boldsymbol{\theta} = \mathbf{X} \bbeta$.  We now turn our attention to Bayesian inference for canonical GLMs, which is the main focus of this paper.
Given a prior distribution $\pi_p$ for $\bbeta \in \mathbb{R}^p$, the scientist constructs the posterior distribution of ${\bbeta}$ given $\mathbf{y}:= (y_1, y_2, \dots, y_n)$ as  
\begin{equation}
\label{eq:posterior} 
    \frac{\mathrm{d} \mu }{\mathrm{d} \pi_p } (\bbeta | \mathbf{y}) :=  \frac{1}{\mathcal{Z}_p} \exp \left [ \sum_{i = 1}^{n} \left ( y_i \theta_i -b(\theta_i) \right ) \right ]  \text{,} 
\end{equation}
where $\mathcal{Z}_p$ is the normalizing constant of the posterior distribution, i.e.,
\begin{equation}
\label{eq:partition_function}
    \mathcal{Z}_p := \int \exp \left [ \sum_{i = 1}^{n} \left (  y_i \theta_i -b(\theta_i) \right ) \right ]   \mathrm{d} \pi_p (\bbeta) \text{.}
\end{equation}
Borrowing terminology from statistical physics, we will refer to $\mathcal{Z}_p$ as the partition function of the model and
\begin{equation}
\label{eq:hamiltonian}
    H(\bbeta) :=\log L(\bbeta)=\sum_{i = 1}^{n}   \left (  y_i \theta_i - b(\theta_i) \right ) 
\end{equation}
as the Hamiltonian. From a theoretical perspective, one seeks to understand the properties of the posterior distribution $\mu(\cdot|\mathbf{y})$. To this end, it is crucial to first analyze the properties of the log partition function $\mathcal{Z}_p$. Building up on ideas developed in the study of mean-field spin glasses, the log-partition function of canonical GLMs has been characterized recently in \cite{barbier2019optimal}, under the assumption that the design entries $\mathbf{X}$ are i.i.d., and that the GLM is well-specified i.e., the observed data $\{(y_i, \mathbf{x}_i): 1\leq i \leq n\}$ is actually sampled from a model of the form \eqref{eq:NEF}. This well-specified property leads to powerful symmetries in the posterior e.g. the Nishimori identities (see \cite{barbier2019optimal} for additional details), which are crucial in the subsequent analysis. In sharp contrast, we will allow both deterministic and random designs, and will not assume a well-specified model. Thus our results are incomparable to the ones derived in \cite{barbier2019optimal}. In addition, we note that our results are not applicable to the proportional asymptotic setting with an i.i.d. design, and thus perfectly complements the existing results.



In this work, we analyze the log-partition function $\log \mathcal{Z}_p$ and the posterior distribution $\mu(\cdot|\mathbf{y})$ under relatively weak assumptions on the design $\mathbf{X}$ and the prior $\pi_p$. To this end, we make the following assumptions. 
\begin{assumption}[Properties of the prior $\pi_p$] 
\label{assump:prior} 
We work under a product prior with bounded support. Specifically, we assume that $\pi_p:= \pi^{\otimes p}$ for a probability distribution $\pi$ on $[-1,1]$. 
\end{assumption} 
The interval $[-1,1]$ is not special---our results will generalize immediately to any compact interval $[-M,M]$. Next, we turn to the assumption on the non-linearity $b(\cdot)$. Note that $b(\cdot)$ as defined in \eqref{eq:b} is the cumulant generating function of the probability distribution $\lambda$ on $\mathcal{X}$, and thus $b''(\cdot) \geq 0$. 
We always assume the following smoothness conditions on the derivatives of $b(\cdot)$. 
\begin{assumption}[Properties of $b''$]\label{assump:cont} Assume that the function  $b''(\cdot) : \mathbb{R} \to \mathbb{R}$ is bounded and uniformly continuous. 
\end{assumption}
We note that the linear regression model and both logistic regression models as defined in Examples~\ref{def:linear_regression} and \ref{def:logistic_regression} satisfy Assumption~\ref{assump:cont}. Finally, we collect our assumptions on the design $\mathbf{X}$.

\begin{assumption}[Assumptions on $\mathbf{X}$] 
\label{assump:design} 
Assume that the design $\mathbf{X}$ satisfies the following conditions. 
    \begin{itemize}
        \item[(i)] $\|\mathbf{X}^{\top} \mathbf{X}\|_2 = O(1)$. 
        \item[(ii)] For any $\delta>0$ we have  $\sum_{i=1}^{n} \sum_{j=1}^{p} x_{ij}^{2} \mathbf{1}(|x_{ij}| > \delta) = o(p)$. 

        \item[(iii)] Finally assume that 
        \begin{align}
         \lim_{C \to \infty}  \lim_{p \to \infty}  \frac{1}{p} \sup_{S \subset [n], |S| \leq Cp} \sum_{i\in S} \|\mathbf{x}_i\|_2^2 =0. \label{eq:norm} 
        \end{align}
    \end{itemize}
\end{assumption}

Before proceeding further, we pause to interpret Assumption~\ref{assump:design}.  Assumption~\ref{assump:design} (i) is essentially a convenient normalization, and ensures that $\log \mathcal{Z}_p$ is of order $p$. This assumption already appears in the literature on high dimensional Bayesian linear regression \cite{mukherjee2021variational,celentano2023mean}. Assumption~\ref{assump:design} (ii) ensures that the design does not have too many $O(1)$ entries. To understand the third assumption, first note that for any $S \subseteq [n]$,  the quantity $\sum_{i \in S} \|\mathbf{x}_i\|_2^2$ represents the Frobenius norm of the sub-matrix $\mathbf{X}_S = (x_{ij}: i \in S, j \in [p])$. Thus, LHS of \eqref{eq:norm} can be bounded as follows: $$\mathrm{Tr}(\mathbf{X}_{S}^{\top} \mathbf{X}_S)\le \mathrm{Tr}(\mathbf{X}^\top \mathbf{X})\le p \|\mathbf{X}\|_2^2=O(p).$$
The third assumption demands that $\mathrm{Tr}(\mathbf{X}_{S}^{\top} \mathbf{X}_S) = o(p)$ uniformly over $|S|\le Cp$.

The log-partition function $\log \mathcal{Z}_p$ is generally intractable, in the absence of specific assumptions on the design $\mathbf{X}$ and the generative model on $\mathbf{y}$. We will employ the notion of the naive mean-field approximation to approximate the log normalizing constant. For any distribution $Q \ll \pi_p$, it follows by direct computation that 
\begin{align}
\log \mathcal{Z}_p = \mathbb{E}_{Q}[H(\boldsymbol{\beta})] - \mathrm{D}_{\mathrm{KL}} (Q \| \pi_p) - \mathrm{D}_{\mathrm{KL}} (Q\| \mu).  \label{eq:ELBO1}
\end{align} 
Using the non-negativity of KL-divergence, we have,  
\begin{equation}
        \log \mathcal{Z}_p \geq   \sup_{Q\in \mathcal{P}} \Big\{ \mathbb{E}_{\bbeta \sim Q} \left [  H (\bbeta) \right ] -\operatorname{D}_{\text{KL}}(Q \| \pi_p )\Big\}, \label{eq:ELBO}
\end{equation}
 where $\mathcal{P}$ is the set of all product measures $Q = \prod_{i=1}^{p} Q_i$ on $[-1,1]^p$ which are absolutely continuous with respect to $\pi_p$. The bound \eqref{eq:ELBO} is generally only a lower bound. We identify a broad set of conditions which ensure that the lower bound is essentially tight (to leading order). To this end, we introduce the following definition. 

\begin{definition}
    We say that the Naive Mean-Field (NMF) approximation is correct to leading order if 
    \begin{align} 
 \log \mathcal{Z}_p = \sup_{Q\in \mathcal{P}} \Big\{ \mathbb{E}_{\bbeta \sim Q} [H (\bbeta)]  -\operatorname{D}_{\mathrm{KL}}(Q \| \pi_p )) \Big\} +o(p), \label{eq:nmf_approx} 
    \end{align} 
    Using \eqref{eq:ELBO1}, the above definition is equivalent to
  \begin{align*}
    \inf_{Q\in \mathcal{P}} \mathrm{D}_{\mathrm{KL}} (Q\| \mu)=o(p).
    \end{align*}
\end{definition}

\noindent 
Our first result identifies conditions on the design $\mathbf{X}$ such that \eqref{eq:ELBO} is correct to leading order (i.e.~\eqref{eq:nmf_approx} holds). We collect some notions that are crucial for our first result. 

\begin{definition}
For any $\bbeta\in [-1,1]^p$, define the matrix $\mathbf{A}_{\bbeta} := \mathbf{X}^{\top} \mathbf{D}_{\bbeta} \mathbf{X} - \mathrm{diag}(\mathbf{X}^{\top} \mathbf{D}_{\bbeta} \mathbf{X})$, where $\mathbf{D}_{\boldsymbol{\beta}} \in \mathbb{R}^n$ is a diagonal matrix with $\mathbf{D}_{\boldsymbol{\beta}}(i,i) := b''(\langle \mathbf{x}_i, \boldsymbol{\beta}\rangle)$.

\end{definition}

\begin{definition}
Given a set $S\subseteq \R^k$ and $\delta>0$, we say a set $\widetilde{S}\subseteq\R^k$ is a $\delta$ net with respect to a metric $d(\cdot,\cdot)$ on $\R^k$, if for every ${\bf v}\in S$ there exists $\widetilde{\bf v}\in \widetilde{S}$ such that $d({\bf v},\widetilde{\bf v})\le \delta$.
\end{definition}
\begin{theorem}
\label{thm:covering}
Suppose Assumptions \ref{assump:prior}, \ref{assump:cont} and \ref{assump:design} hold. 
     Finally, assume that for any $\varepsilon>0$, the set
     $$\mathcal{A}_p:=\{ \mathbf{A}_{\bbeta} \bgamma, \bbeta, \bgamma\in [-1,1]^p\}\subseteq \R^p$$
     has a $p\varepsilon$ net in $\ell_1$ metric of size $N(p,\varepsilon)$, such that
%
%
    $$\lim_{p \to \infty} \frac{1}{p} \log N(p,\varepsilon) =0.$$
    Then the NMF approximation is correct to leading order. 
\end{theorem}


In this context, one might be naturally curious about the necessity of our assumptions on the design $\mathbf{X}$. 
Specifically, one wonders if Assumption~\ref{assump:design} may be weakened. In fact, there are well-known settings where Assumption~\ref{assump:design} is violated, and the NMF approximation fails: a prominent example arises in the proportional asymptotics regime (i.e., $n \propto p$) and $x_{ij} \sim \mathcal{N}(0,\frac{1}{n})$. Assumption~\ref{assump:design} (iii) is violated in this example. Non-rigorous predictions from spin-glass theory predict that the NMF approximation should be inaccurate in this regime; instead, the Thouless-Anderson-Palmer (TAP) approximation is expected to yield a tight estimate for the log-partition function \cite{mezard1987spin}. Another prominent counterexample arises for sparse designs: assume $n \propto p$ and $x_{ij} \sim \mathrm{Ber}(1/n)$ are i.i.d. The design violates both Assumption~\ref{assump:design} (ii) and (iii). One expects the Bethe approximation to yield a tight bound to the log-partition function in this example \cite{mezard2009information}. In conclusion, it is impossible to weaken Assumption~\ref{assump:design} without imposing alternate conditions on the design matrix $\mathbf{X}$.  

The net condition in Theorem \ref{thm:covering} can be difficult to check in specific applications. We next provide a set of easily checkable sufficient conditions. 

\begin{theorem}
\label{thm:covering_suff} 
Suppose Assumptions \ref{assump:prior}, \ref{assump:cont} and \ref{assump:design} hold.  In addition, assume that both the following conditions are satisfied:

\begin{itemize}
\item[(i)] For any $\varepsilon>0$, the set of matrices  $\{\mathbf{A}_{\bbeta}, \bbeta \in [-1,1]^p\}$ has  an $\varepsilon$-net in operator norm of size $N(p,\varepsilon)$ such that 
$$\lim_{p\to\infty}\frac{1}{p}\log N(p,\varepsilon)=0.$$
\item[(ii)] For any $\bbeta \in [-1,1]^p$ we have $\mathrm{Tr}(\mathbf{A}_{\bbeta}^2) = o(p)$. 
\end{itemize} 
Then the NMF approximation is correct to leading order. 
\end{theorem} 

Even though condition (i) in Theorem \ref{thm:covering_suff} replaces one net condition by another, this is significantly more useful in practice, as illustrated by the applications in 
the following corollary. 

\begin{corollary}
\label{cor:applications} 
Suppose Assumptions \ref{assump:prior}, \ref{assump:cont} hold. Then the NMF approximation is correct to leading order if any one of the conditions (i), (ii) or (iii)  are satisfied in addition.  
\begin{itemize} 
\item[(i)] Suppose Assumption~\ref{assump:design} holds. In addition, for any $C>0$, as $p \to \infty$, $$\lim_{p \to \infty} \max_{S \subset [n]: |S| \leq Cp} \Big\| \sum_{i \in S} \mathbf{x}_i \mathbf{x}_i^{\top}  \Big\|_{\mathrm{op}}=0.$$ Finally, assume that $\mathrm{Tr}(\mathbf{A}_{\bf 0}^2) = o(p)$ (where $\mathbf{A}_0$ is just $\mathbf{A}_{\bbeta}$ with $\bbeta={\bf 0}$).
\item[(ii)] Assume that the rows of $\mathbf{X}$ are i.i.d.~centered gaussian, i.e.~$\mathbf{x}_i \sim \mathcal{N}(0, \boldsymbol{\Sigma}_p/n)$. In addition, assume that $n \gg p$, $\|\bSigma_p\|_{\mathrm{op}} \leq C$ for some universal constant $C>0$ and $\mathrm{Tr}(\bSigma_{p,\mathrm{off}}^2)=o(p)$, where $\bSigma_{p,\mathrm{off}} = \bSigma_p - \mathrm{diag}(\bSigma_p)$.  
\item[(iii)] Suppose Assumption~\ref{assump:design} holds and  $p/n \to \kappa \in (0, \infty)$. 
\end{itemize} 
\end{corollary} 


\noindent 
We now turn to the practical implementation of the NMF approximation. 
In practice, one hopes to approximate the posterior distribution of interest by a product distribution---to this end, one employs the lower bound \eqref{eq:ELBO}, and hopes to compute an approximate optimizer of the RHS. In the context of the linear model, the hamiltonian $H(\bbeta)$ is a quadratic function of $\bbeta$---consequently, one can explicitly identify the class of product distributions $Q= \prod_{i=1}^{p} Q_i$ which maximize the RHS. In turn, this enables one to develop fast numerical schemes to compute the tightest lower bound to $\log \mathcal{Z}_p$ (see the analysis in \cite{mukherjee2021variational} for additional details in the context of the linear model). This approach underlies popular Variational Inference (VI) schemes employed in diverse data applications (see e.g. the recent survey \cite{blei2017variational}). However, this key property is lost as soon as one moves beyond the linear model setting. This is a key barrier in the development of Variational inference algorithms for GLMs such as logistic regression. This issue has been identified as a key challenge in VI, and has been emphasized repeatedly in the literature \cite{blei2017variational,paisley2012variational,wang2013variational}. Some algorithmic solutions have been developed for special cases (e.g. for logistic regression with gaussian prior \cite{jaakkola2000bayesian}), but general solutions are absent in the literature. 


Our next result directly addresses this challenge. Under Assumptions \ref{assump:prior}, \ref{assump:cont} and \ref{assump:design}, we identify novel classes of parametric sub-families such that these sub-families contain a sequence of approximate optimizers of \eqref{eq:ELBO}. In turn, this directly facilitates the development of new VI algorithms---we discuss the state-of-the-art, and our contributions in this direction in Section \ref{sec:algo}.  

\begin{definition}[Exponential tilting]
\label{def:exp_tiltting}
For any $\gamma := (\gamma_1, \gamma_2) \in \mathbb{R} \times [0,\infty)$ and probability distribution $\pi$ on $[-1,1]$, we define the tilted measure $\pi_{\gamma}$ as
\begin{align*}
    \frac{\mathrm{d} \pi_{\gamma}}{\dif \pi}(x) := \exp \left ( \gamma_1 x - b''(0) \frac{\gamma_2}{2}x^2 - c_{\pi}(\gamma)\right ),
    \text{    }
    c_{\pi}(\gamma) := \log \int_{[-1,1]}\exp \left ( \gamma_1 x - b''(0) \frac{\gamma_2}{2}x^2 \right ) \mathrm{d} \pi(x).
\end{align*}
For any $d>0$, the function $c_\pi(\cdot,d)$ is strictly convex, and so the function $\dot{c}_\pi(\cdot,d)$ is strictly increasing. Here, and everywhere else, by $\dot{c}_\pi$ we mean the derivative of $c_\pi(\cdot,\gamma_2)$. Let  $h: [-1,1] \times [0, \infty) \to \mathbb{R}$ denote the inverse of the function $\dot{c}_\pi(\cdot,d)$, defined by 
\begin{equation*}
    \begin{aligned}
        u =\dot{c}_{\pi}(h(u,d),d)= \mathbb{E}_{X\sim \pi_{(h(u,d),d)}}[X].  
    \end{aligned}
\end{equation*}
\end{definition}

\noindent
We now turn to our next main result. 

\begin{theorem}[Algorithmic implications]
\label{thm:alg_glm}
\begin{itemize}
\item[(i)]
Suppose that the assumptions of Theorem \ref{thm:covering} hold, and assume further that
$$\| \mathbf{X}^{\top} (\mathbf{y} - b'(0) \boldsymbol{1})\|^2 =O(p).$$  Then we have
\begin{align}
\label{eq:alg_1}
    \log \mathcal{Z}_p -  \sup_{Q \in \Gamma} \left\{ \mathbb{E}_{\bsigma \sim Q} \left [  H (\bsigma) \right ] -\operatorname{D}_{\text{KL}}(Q \| \pi_p )\right\}   = o(p),
\end{align}
where $\Gamma:=\Big \{ \prod_{j=1}^p \pi_{( \gamma_j,d_j  ) } : \gamma \in \mathbb{R}^p \Big \}$ is the collection of all exponential tilts of the prior.

\item[(ii)]  Under the same assumptions as part (i) above, 
we have
    \begin{align}
    \label{eq:alg_3}
        \log \mathcal{Z}_p -  \sup_{\mathbf{u} \in [-1,1]^p } &\Big \{  \mathbb{E}_{\bsigma \sim Q_{\mathbf{u}} }[ H(\bsigma)] - D_{\mathrm{KL}}(Q_{\mathbf{u}} \| \pi_p) \Big) \Big\} = o(p),  
    \end{align}
    where $Q_{\mathbf{u}} = \prod_{j=1}^p \pi_{(h(u_j,d_j),d_j)}$. 
\end{itemize}
\end{theorem}

\noindent 
In general, the high-dimensional posteriors under consideration have subtle dependencies, and understanding their probabilistic properties (e.g. the low dimensional marginals) can be challenging. Our next result illustrates the usefulness of the NMF approximation in this regard---in special cases, the NMF approximation yields explicit characterizations of the low-dimensional marginals of the high-dimensional posterior distribution. We expect such guarantees to be particularly useful for downstream Bayesian inference---this has been illustrated in the context of Bayes linear regression in \cite{mukherjee2021variational}; we expect similar applications in the setting of GLMs.

\begin{definition}
\label{def:well_separated} 
We say that there exists a well-separated optimizer $\mathbf{u}^* \in [-1,1]^p$ if for all $\delta>0$, there exists $\varepsilon>0$ such that 
\begin{align}
& \sup_{\mathbf{u} \in [-1,1]^p:  \| \mathbf{u} - \mathbf{u}^*\|^2 > p\delta } \bigg \{  \mathbb{E}_{\bsigma \sim Q_{\mathbf{u}} }[ H(\bsigma)] - D_{\mathrm{KL}}(Q_{\mathbf{u}} \| \pi_p) \bigg\}  
\leq&  \sup_{\mathbf{u} \in [-1,1]^p } \bigg (  \mathbb{E}_{\bsigma \sim Q_{\mathbf{u}} }[H(\bsigma)] - D_{\mathrm{KL}}(Q_{\mathbf{u}} \| \pi_p) \bigg)  - p\varepsilon. \nonumber 
\end{align} 
\end{definition} 

Finally, we will require the notion of Wasserstein or transport distance (see e.g. \cite{panaretos2019wasserstein} for an in-depth introduction). 

\begin{definition}[Wasserstein distance] 
\label{def:wasserstein}
Let $\mu$, $\nu$ be two probability distributions on $[-1,1]^{p}$. Define 
\begin{align}
d_{W_1}(\mu, \nu) = \inf_{\Gamma} \Big\{  \| \mathbf{X} - \mathbf{Y} \|_1 : (X,Y) \sim \Gamma \Big\}, \nonumber 
\end{align} 
where $\Gamma$ represents a coupling of $(\mu, \nu)$. 
\end{definition}

\begin{theorem}[Posterior structure] 
\label{thm:structure} 
Assume that the conditions of Theorem \ref{thm:covering} are satisfied. Further, assume that there exists a well-separated optimizer $\mathbf{u}^*$. Define $\tau_j^* = h(u_j^*,d_j)$. Then $$d_{W_1}\Big(\mu, \prod_{j=1}^{p} \pi_{(\tau_j^*,d_j)} \Big) =o(p).$$
\end{theorem}

\noindent 
Our final corollary illustrates some statistical implications of our structure result Theorem~\ref{thm:structure}. 

\begin{corollary}
	[Logistic Regression: coverage rate and prediction error]
    \label{cor:logistic_regression_classification_error}
    Suppose we work with the Logistic Regression model as defined in \eqref{eq:logistic_regression_def}. Assume all conditions in  Theorem~\ref{thm:structure} are satisfied. 
    \begin{itemize}
    	\item[(i)] 
    		For $j \in [p], \varepsilon > 0,$ and $ \alpha \in (0,1/2)$, let
    		\begin{align}
    			\mathcal{I}_{j}^{\varepsilon} := (q_j^{\alpha/2} - \varepsilon, q_j^{1 - \alpha /2}+ \varepsilon), \nonumber
    		\end{align}
    		where $q_j^{\alpha/2}$ and $q_j^{1 - \alpha /2}$ are the $\alpha /2$ and $1 - \alpha /2 $ quantiles of $\pi_{(\tau^*_j, d_j)}.$ Then we have as $p \to \infty$,
    		\begin{align}
    			\mu \Big ( \frac{1}{p}\sum_{j = 1}^p \mathbb{I}(\beta_j \in \mathcal{I}_j^{\varepsilon}) \geq 1 - \alpha - \epsilon  \Big| \mathbf{y} \Big ) \to 1. \nonumber
    		\end{align}
		\item[(ii)] Let $ \tilde{\mathbf{x}} = (\tilde{x}_1, \tilde{x}_2, \dots, \tilde{x}_p) \in \mathbb{R}^p$ be a new data point. Suppose there exists $C >0$ such that $p \|   \tilde{\mathbf{x}} \|_{\infty} < C$. Then as $p \to \infty$,
			\begin{align}
		    	\mu \Big (  \Big | \sum_{j=1}^p \beta_j \tilde{x}_{j} - \sum_{j=1}^p \mathbb{E}_{\beta \sim \pi_{ ( \tau^*_j, d_{j} )}}  [ \beta] \tilde{x}_j   \Big | > \epsilon \Big| \mathbf{y}  \Big ) \to 0. \nonumber
		    \end{align}
		    In addition, suppose $ \tilde{Y} \sim \operatorname{Bern}(\tilde{f}(\tilde{\mathbf{x}} ))$, $\hat{Y}(\bbeta) \sim \operatorname{Bern} \left ( \phi ( \tilde{\mathbf{x}}^\top \bbeta ) \right )$, where $\phi(t):= \frac{1}{1 + e^{-t}}$ and $\tilde{Y} \perp \!\!\! \perp \hat{Y} (\bbeta)$, then the classification error can be characterized as
		    \begin{align}
		    	\mathbb{P}_{ \tilde{Y},\, \bbeta \sim \mu_{\bbeta| \mathbf{y}}, \hat{Y}(\bbeta)}(  \tilde{Y}  \ne \hat{Y}(\bbeta) ) =  2 \cdot \Big | \phi \Big ( \sum_{j=1}^p \mathbb{E}_{\beta \sim \pi_{ ( \tau^*_j, d_{j})}}  [ \beta] \tilde{x}_j    \Big)    - \tilde{f}(\tilde{\mathbf{x}})  \Big |   + o_{\mu(\cdot | \mathbf{y})}(1). \nonumber
		    \end{align}
    \end{itemize}
\end{corollary}

\noindent
\textbf{Notations:} 
We use the usual Bachmann-Landau notation $O(\cdot)$, $o(\cdot)$, $\Theta(\cdot)$ for sequences. For a sequence of random variables $\{X_p : p \geq 1\}$, we say that $X_p = o(1)$ if $X_p \stackrel{P}{\to} 0$ as $p \to \infty$ and $X_p = o(f(p))$ if $X_p /  f(p) = o(1)$. Similarly, with a slight abuse of notation, we say that $X_p = o_{\mu}(1)$ if $X_p \stackrel{\mu}{\to} 0$. Throughout, we use $C, C_1, C_2 \cdots $ to denote positive constants independent of $n,p$. Further, these constants can change from line to line. For any square symmetric matrix $A$, $\|A\|_\text{op}$ and $\|A\|_F$ denote the matrix operator norm and the Frobenius norm respectively.

\subsection{Prior Work}
We highlight related prior work, and discuss the connections with our results in this section.  

\begin{itemize}
\item[(i)] \textbf{Non-linear large deviations:} The theory of non-linear large deviations was initiated in the seminal work \cite{MR3519474}, as a general tool to solve the large deviation problem for sub-graph counts on sparse random graphs. Several alternative approaches to a general theory have emerged since, including the gaussian-width based approach of \cite{MR3881829}, and the information theoretic treatment by \cite{austin2019structure}. We exploit the approach of \cite{austin2019structure} in this work; however, the alternative approaches are closely related, and we believe it should be possible to derive similar results using any of the alternative approaches. In \cite{basak2017universality,jain2018mean,jain2019mean} the authors utilize these general ideas to analyze the log-partition function for Ising/Potts models on graph sequences with growing degrees. 

In \cite{mukherjee2021variational}, the first and third author employ ideas from non-linear large deviation theory to study high-dimensional bayesian linear regression. The current paper is a continuation of this research direction; we extend the prior results significantly beyond the linear model in this paper. As we emphasize in Section~\ref{sec:technical_contributions}, the linear model was closely related to prior analyses of Ising/Potts models, due to the quadratic nature of the associated likelihood. The canonical GLM has a very different structure, and thus requires significantly different ideas. We emphasize our main technical contributions, and highlight the difference compared to existing works in Section~\ref{sec:technical_contributions}.



\item[(ii)] \textbf{Spin glasses and variational inference:} There has been a rich-thread of recent research which utilizes ideas from mean-field spin glass theory to analyze high-dimensional Bayesian inference problems. These analyses typically assume that $x_{ij} \sim \mathcal{N}(0,1/n)$ are i.i.d. and work under a proportional asymptotic regime i.e., $p \propto n$. In particular, \cite{barbier2020mutual} characterizes the limiting mutual information for Bayes linear regression, while the seminal work \cite{barbier2019optimal} characterizes the limiting mutual information and estimation errors for GLMs. The ``adaptive interpolation method" has emerged as a central tool in this endeavor, and facilitates the aforementioned analyses. We note that existing works typically assume a well-specified setting i.e., the data is generated from the model family being fitted to the data. In this setting, one can utilize the powerful Nishimori identities \cite{nishimori2001statistical} to establish replica-symmetry of the posterior distribution. Going beyond the well-specified setting is generally challenging---the associated posterior is expected to be full-replica symmetry breaking in many cases. We highlight exciting recent progress in \cite{barbier2021performance,barbier2022strong}, where the authors establish replica symmetry even beyond the well-specified setting, by exploiting log-concavity of the prior.  

Finally, we note that our analysis is based on the accuracy of the NMF approximation. The NMF approximation is expected to be inaccurate for regression models in the proportional asymptotics regime. Instead, the Thouless-Anderson-Palmer (TAP) approximation from spin glasses is conjectured to yield a tight approximation to the log-partition function \cite{krzakala2014variational}. The TAP representation for the log-partition function of bayes linear regression with a uniform spherical prior was established at high-temperature in \cite{qiu2023tap}. The instability of NMF free energy has been established for the related spiked matrix problem in  \cite{ghorbani2019instability}. In follow up work, \cite{fan2021tap} and \cite{celentano2023local} studies the TAP free energy, and establishes asymptotic bayes optimality for the statistical procedure obtained.



\item[(iii)] \textbf{NMF Variational Bayes:} Variational Bayesian inference, based on the NMF approximation, has attracted significant attention in high-dimensional Bayesian statistics. We refer to \cite{MR3263115}, \cite{MR3709863}  and the references therein for some early works, and to \cite{MR2221291,wang2019frequentist,wang2019variational} for theoretical analyses of NMF Variational Bayes under classical low-dimensional asymptotics. More recently, the community has focused on the statistical performance of variational posteriors in high-dimensional models, often with additional sparsity assumptions on the underlying model; we refer the interested reader to \cite{MR3595173,MR4124331,han2019statistical,MR4480711,ray2020spike,MR4102680, qiu2023sub} and references therein for the state-of-the-art in this area. In a related direction, \cite{katsevich2023approximation} investigates the theoretical properties of Gaussian Variational Inference in high-dimensions.



\item[(iv)] \textbf{Algorithms:} The main attraction of Variational Inference from a practical standpoint lies in the dramatic computational gains conferred by these algorithms. In special cases, e.g. Bayes linear regression with product priors, one typically employs Coordinate Ascent Variational Inference (CAVI) algorithms to find an optimizer $\hat{Q}$ of  the RHS of \eqref{eq:ELBO}. CAVI algorithms are typically fast and iterative, and can be applied to very large problems, well-beyond the purview of traditional MCMC methods. However, such algorithms are generally unavailable for Bayesian GLMs. In a celebrated work, \cite{jaakkola2000bayesian} introduced an algorithm for Bayesian logistic regression with gaussian priors --- this algorithm introduces a quadratic lower bound to the logistic likelihood, and then optimizes this lower bound. This algorithm was adopted to the high-dimensional logistic regression model with gaussian spike and slab priors recently in \cite{ray2020spike}. These algorithms do not generalize beyond this particular example, and thus there have been many attempts to design general purpose Variational Inference algorithms for Bayesian GLMs. A recent general idea ``black box variational inference" \cite{paisley2012variational, ranganath2014black} --- in this case, the statistician searches for the optimizer $\hat{Q}$ in \eqref{eq:ELBO} via direct gradient descent. The gradient is typically computed by Monte Carlo, or quasi Monte Carlo \cite{liu2021quasi}. We refer the reader to \cite{blei2017variational} for a discussion of recent developments in this direction. Finally, various message passing based algorithms have also been proposed for this problem \cite{knowles2011non,nolan2017accurate} ---unfortunately, these algorithms usually lack theoretical guarantees, which restricts their use in practice. 

\item[(v)] \textbf{Variational inference via Optimal Transport:} There has been significant recent progress in the design and analysis of variational inference algorithms using ideas originating from optimal transport \cite{santambrogio2015optimal}. \cite{lambert2022variational} establishes convergence of CAVI algorithms under log-concavity,  \cite{diao2023forward} and \cite{jiang2023algorithms} develop algorithms to solve Gaussian VI and mean-field VI assuming smoothness on the posterior,  \cite{lacker2023independent} analyzes mean-field Variational inference for diffusion processes, \cite{arnese2024convergence} derives guarantees for the CAVI algorithm for log-concave measures, \cite{yao2024wasserstein} develops a Wasserstein gradient descent algorithm for mean-field Variational inference. 


\end{itemize}

\subsection{Technical Contributions:} 
\label{sec:technical_contributions} 
We highlight our main technical contributions in this section.

\begin{itemize} 
 \item[(i)] \textbf{NMF for non-linear statistical models:} The theory of non-linear large deviations was originally developed to characterize large deviation probabilities for sub-graph counts in sparse random graphs \cite{MR3881829,MR3519474,yan2020nonlinear}. These techniques were used to analyze Ising and Potts models on random graphs first in \cite{basak2017universality}; the initial results were sharpened in several follow up works \cite{MR3881829,eldan2020taming,jain2019mean,jain2018mean}. In the context of statistical learning problems, two authors of this article utilized these techniques to analyze the bayesian linear regression problem in \cite{mukherjee2021variational}. The Ising/Potts model and bayes linear regression analyses crucially utilize the quadratic nature of the associated Hamiltonians. To highlight the main idea at a high-level, fix any symmetric matrix $\mathbf{A} \in \mathbb{R}^{p \times p}$ with $A_{i,i}=0$ for $1 \leq i \leq p$. Consider the quadratic Hamiltonian $H(\bbeta) = \bbeta^{\top} \mathbf{A} \bbeta$; an application of the general framework developed in \cite{austin2019structure} yields that NMF is correct to leading order if the set $\{\mathbf{A} \bgamma : \bgamma \in [-1,1]^{p} \}$ has an $L^1$ net of size $\exp(o(p))$. In prior works, the authors identify the condition $\mathrm{Tr}(\mathbf{A}^2) = o(p)$ as a sufficient condition to guarantee the existence of an efficient net. The success of the first step relies crucially on the quadratic nature of the Hamiltonian, and does not generalize directly to settings beyond linear regression. 

To analyze non-linear statistical models using existing ideas, one might first try a global second order Taylor expansion: 
\begin{align}
H(\bbeta) \approx H(\bbeta_0) +  (\bbeta - \bbeta_0 )^{\top} \nabla H(\bbeta_0)  + \frac{1}{2}  (\bbeta - \bbeta_0 )^{\top} \nabla^2 H(\bbeta_0)  (\bbeta - \bbeta_0 ), \nonumber 
\end{align}
where $\bbeta_0$ is an arbitrary reference point in $[-1,1]^{p}$. If the  difference between the two sides above is $o(p)$ uniformly in $\bbeta$, one can a priori replace the original posterior distribution by a pseudo Gibbs measure with the approximate quadratic Hamiltonian, and then invoke existing ideas for quadratic models. This route can indeed be formalized in special cases: for example, if the design entries are i.i.d. $\mathcal{N}(0,1/n)$ and $p=o(n)$, one can choose $\bbeta_0 = \mathbf{0}$ and rigorously justify the error of the global quadratic expansion using standard operator norm bounds from random matrix theory. 

However, one can easily construct examples of design matrices $\mathbf{X}$ where the global quadratic expansion strategy does not work: a simple example in this regard is the case $n=p$ and $x_{ij} = 1/p$ for all $1\leq i, j \leq p$. In this case, the NMF approximation should be valid, as the Hamiltonian is a function of a single projection $\bar{\bbeta} = \frac{1}{p} \sum_{i=1}^{p} \beta_i$, but the Hamiltonian cannot be globally approximated by a quadratic function. The situation is considerably more involved if the design has two parts, one similar to the random design example described above, and the other corresponding to the structured, low-rank structure described here. 

In Theorems \ref{thm:covering} and \ref{thm:covering_suff}, we derive easily verifiable sufficient conditions for the correctness of the NMF approximation in the special context of canonical GLMs. Our conditions can cover both the random design and low-rank examples described above, and any combination thereof. To derive these sufficient conditions, we employ the general framework introduced in \cite{austin2019structure}; we establish that even though a global quadratic expansion is possibly incorrect, one can still approximate the Hamiltonian ``locally" by a quadratic function in our setting. The NMF approximation is then true if (i) the number of distinct local quadratic approximations required is of size $\exp(o(p))$ and (ii) the local Hessians are each dominated by a ``few" large eigenvalues (this is an informal version of Theorem \ref{thm:covering_suff}). When these conditions hold, we informally think that the model has ``low Hessian complexity". This low Hessian complexity-based strategy to establish the correctness of the NMF approximation should be valid in settings beyond GLMs, and our proofs could thus be of independent interest. 

We note that the importance of ``Hessian complexity" in the context of the NMF approximation for the log-partition function has been hinted at in recent prior work \cite{barbier2022strong}; however, this work focused on log-concave distributions, and thus the approach and main results are independent of our work. The log-concave case is potentially conceptually simpler, as the Gibbs measure is always in a ``pure state". However, we find this connection tantalizing, and believe exploring further connections in this realm is a fruitful direction for future research. 

\item[(ii)] \textbf{Identification of sufficient parametric sub-families:} For models with quadratic Hamiltonians, one can easily establish that any NMF optimizer must have a special structure; indeed, each component distribution must be a quadratic tilt of the base measure $\pi$. This explicit identification has played a crucial role in prior analyses of Ising/Potts and linear regression models \cite{basak2017universality, mukherjee2021variational}. This ingredient is absent for GLMs (and more generally for models with non-quadratic hamiltonians), which a priori makes their analysis intractable. Building on the analysis of \cite{austin2019structure}, Theorem~\ref{thm:alg_glm} identifies appropriate sub-families of product distributions which suffice for NMF approximation in GLMs.  This observation has several distinct implications. First, it directly helps in our subsequent investigations into the probabilistic properties of the posterior distribution (see Theorem~\ref{thm:structure}). Second, this finding can aid in implementing NMF approximations algorithmically in practical data analysis. We discuss this impact in detail in Section~\ref{sec:algo}.  

\item[(iii)] \textbf{Low-dimensional marginals of the posterior distribution:} If the NMF approximation is correct to leading order, the associated Gibbs measure can generally be approximated by a mixture of product distributions (with $\exp(o(p))$ many mixture components) \cite{austin2019structure}. In general, it is challenging to pin point the mixture components and their weights under the Gibbs distribution; as a consequence, it is often challenging to characterize natural properties of the Gibbs distribution e.g., the low-dimensional marginals, Mean-square error, average coverage guarantees etc. For well-specified models (i.e., if the data is sampled from the fitted family), the Nishimori identities and the associated strong replica symmetry properties provide powerful additional ingredients to characterize high-dimensional posteriors. However, beyond the well-specified setting, these properties are generally absent; in a recent work authors of \cite{barbier2020strong} establishes that replica symmetric behavior can be recovered in the general setting, provided the prior is log-concave. We emphasize that we never make any assumptions on the true generative model for the response $\mathbf{y}$, and thus our results go beyond the well-specified setting. In this generality, it is quite challenging to analyze the posterior distribution, even under the NMF assumption. Building upon our insights in Theorem \ref{thm:structure}, we identify sufficient conditions for the posterior to be approximable (in Wasserstein metric) to a product measure in Theorem \ref{thm:structure}. The proof of this result is delicate, and constitutes one of our main technical contributions. 

To describe our technical contribution, we introduce some ideas from \cite{austin2019structure} at a high-level; the main conceptual ideas in alternative NMF approximation results are similar, and our approach should be adaptable to these alternate frameworks in a straight-forward manner. Given the assumptions of Theorem~\ref{thm:covering}, \cite{austin2019structure} constructs a partition of $[-1,1]^{p} = \cup_{j=1}^{N} \mathcal{C}_j$, and establishes that the posterior distribution $\mu (\cdot| \mathbf{y})$, restricted to any element $\mathcal{C}_j$, is close to a product measure (in Wasserstein distance). In particular, this implies the approximate decomposition $\mu (\cdot | \mathbf{y}) =  \sum_{j=1}^{N} \mu(\mathcal{C}_j | \mathbf{y}) \mu_{|\mathcal{C}_j}(\cdot|\mathbf{y}) \approx  \sum_{j=1}^{N} \mu(\mathcal{C}_j | \mathbf{y}) Q_j$, where $Q_j$ represents a product distribution for $1\leq j \leq N$. At this step, if one can establish that there exists $1\leq j \leq N$ such that $\mu(\mathcal{C}_j | \mathbf{y}) \approx 1$, this would immediately imply $\mu(\cdot|\mathbf{y}) \approx Q_j$. However, this condition is not easily verifiable in practice. Instead, we work with the well-separated optimizer property (see Definition~\ref{def:well_separated}) --- under this property, there exists a subset $\mathcal{S} \subseteq [N]$ and a product distribution $Q$ such that $\mu(\cup_{j \in S} \mathcal{C}_j | \mathbf{y}) \approx 1$ and for each $j \in S$, $Q_j \approx Q$. This ensures that $\mu(\cdot | \mathbf{y}) \approx Q$. We formalize this idea in the proof of Theorem~\ref{thm:structure}. 
\end{itemize}

\noindent
\textbf{Organization:} The rest of the paper is structured as follows. We include some numerical demonstrations of our results, and explore novel variational inference methodology based on our theoretical insights in Section~\ref{sec:algo}. We discuss some directions for future inquiry in Section~\ref{sec:discussion}. We prove our results in Section~\ref{sec:proofs}. Finally, Section~\ref{sec:auxiliary_results} collects some auxiliary results required for our proofs.

\vspace{10pt} 

\noindent
\textbf{Acknowledgments:} SS gratefully acknowledges support from NSF (DMS CAREER 2239234), ONR (N00014-23-1-2489) and AFOSR (FA9950-23-1-0429). SM gratefully acknowledges support from NSF (DMS-2113414).

\section{Numerical experiments}
\label{sec:algo}
We provide some numerical evidence in support of our theoretical results in Section~\ref{sec:normalizing_const_eval}. In Section~\ref{sec:discrete_algo}, we turn our attention to the algorithmic implications of our results. 
Using \eqref{eq:ELBO1}, if $\hat{Q}$ is an optimizer of the RHS in \eqref{eq:ELBO}, $\hat{Q}$ minimizes the KL divergence between the posterior $\mu$ and the class of product measures. This suggests a natural algorithmic strategy---approximate any optimizer $\hat{Q}$, and use $\hat{Q}$ as a proxy for the true posterior $\mu$. This strategy is ubiquitous in modern high-dimensional statistics and Machine Learning, and is usually referred to as NMF Variational Inference (NMF VI). It is challenging to approximate any optimizer $\hat{Q}$ for a general GLM. In Section~\ref{sec:discrete_algo}, we illustrate that our results can provide some natural algorithmic guidelines in this direction.


\begin{algorithm}
\caption{The Jakkola-Jordan algorithm for logistic regression with Gaussian priors}\label{alg:JJ_algo}
\begin{algorithmic}
\Require $\mathbf{X}, \mathbf{y}, \mathbf{u}_0, \Sigma_0, \mathbf{\xi}_0$
\State $t \gets 1$
\While{A suitable convergence criterion is not satisfied}
\State $\Sigma_t^{-1} \gets \Sigma_0^{-1} + 2 \sum_{i=1}^n \lambda(\xi_{t-1,i}) \mathbf{x}_i \mathbf{x}_i^T$ \Comment{$\lambda(x) := \frac{1}{2x} \left ( \frac{1}{1 + e^{-x}} - \frac{1}{2} \right ) $}
\State $\mathbf{u}_{t} = \Sigma_t \left ( \Sigma_0^{-1} \mathbf{u}_0 + \sum_{i=1}^n (y_i - 1/2) \mathbf{x}_i \right )$ 
\For{$i = 1:n$}
\State
        $\xi_{t,i} = \sqrt{\mathbf{x}_i^T (\Sigma_t + \mathbf{u}_t \mathbf{u}_t^T) \mathbf{x}_i}$
    \Comment{optimizing the variational parameters.}
\EndFor
\State $t \gets t+1 $
\EndWhile \\
\Return $\mathcal{N}(\mathbf{u}_{t-1}, \Sigma_{t-1})$
\end{algorithmic}
\end{algorithm}

\subsection{Approximating the log partition function}
\label{sec:normalizing_const_eval}  
We investigate the efficacy of the NMF approximation to the log partition function in problems of moderate size. This complements the asymptotic conclusions derived in Theorems \ref{thm:covering} and \ref{thm:covering_suff} above. The illustration here serves as a proof of concept---we do not streamline the numerical implementations to optimize computational efficiency. To this end, we study a design matrix $\mathbf{X}$ with a block structure. We assume that $n \geq 2p$.    
Denote $\mathbf{X}^{\top}  = [\mathbf{x}_1, \cdots, \mathbf{x}_{n}]$. 
Setting $\boldsymbol{1}_p := (1, \cdots, 1) \in \mathbb{R}^p$, we define $\mathbf{x}_1= \cdots = \mathbf{x}_p = \frac{1}{p} \boldsymbol{1}_p$; in addition, we set $\mathbf{x}_{p+1} = \cdots = \mathbf{x}_{2p} = [ \frac{1}{p} \boldsymbol{1}_{p/2}, -\frac{1}{p} \boldsymbol{1}_{p/2}]$. Finally, for $2p +1 \leq i \leq n$, $\mathbf{x}_i \sim \mathcal{N}(0, \mathbf{I}_p/n)$ are i.i.d. Given the design, we sample $\mathbf{y}$ from a logistic regression model; the coefficients of the regression model are sampled i.i.d. from a prior $\pi$. As a special case, we consider $\pi = \mathcal{N}(0,1)$---note that this setting is not within the strict purview of Theorems~\ref{thm:covering} and \ref{thm:covering_suff}, due to the unbounded support of the gaussian prior. However, this specific setting has been studied in the prior literature, and there exist well-known estimates which serve as a benchmark for our NMF approximation. Moreover, given the fast decay of the tails of the gaussian prior, we believe that this setting is a relatively minor extension of our setup. 

We consider three distinct approximations to the log-partition function $\log \mathcal{Z}_p$. 
\begin{itemize}
\item[(i)] Monte Carlo estimate---We use the popular package {\textsf{STAN}} to draw approximate samples from the posterior distribution. Using these approximate samples, we obtain an estimate for the MSE of the posterior mean. Finally, {\textsf{STAN}} uses a bridge sampling algorithm to estimate the log-partition function. 
\item[(ii)] Jakkola-Jordan algorithm---The Jakkola-Jordan tangent transform method \cite{jaakkola2000bayesian} was developed to approximate the posterior in this specific example. The method approximates the joint distribution by a multivariate gaussian distribution. We refer to the original paper \cite{jaakkola2000bayesian} for a derivation of the approximation scheme; the resulting algorithm is iterative, and updates the mean vector and covariance matrices of the fitted multivariate gaussian distribution. For the convenience of the reader, we present the iterative update scheme in Algorithm \ref{alg:JJ_algo}. Denote the resulting multivariate gaussian distribution as $Q^{\mathrm{JJ}}$. We estimate the log-partition function as $\log \mathcal{Z}_p \approx \mathbb{E}_{Q^{\mathrm{JJ}}}[H(\bsigma)] - \mathrm{D}_{\mathrm{KL}}(Q^{\mathrm{JJ}} \| \pi^{\otimes p})$. Finally, we estimate the MSE using the mean vector of the fitted multivariate gaussian distribution. 
\item[(iii)] NMF estimate---We use the NMF estimate \eqref{eq:ELBO}, and compute a restricted maximum over the class of independent Gaussian distributions 
\begin{align*}
    \Gamma = \left \{ Q_{\mathbf{u}, \mathbf{v}} = \otimes_{i=1}^p \mathcal{N}(u_i, v_i): \mathbf{u} \in \mathbb{R}^p, \mathbf{v} \in \mathbb{R}_+^p \right \}.
\end{align*}
To fit this variational family, we compute the derivatives of 
\begin{align*}
    \mathcal{M}(\mathbf{u},\mathbf{v}) &:=  \mathbb{E}_{ \bsigma \sim Q_{\mathbf{u}, \mathbf{v}}} \left [ - H (\bsigma) \right ] -\operatorname{D}_{\text{KL}}(Q \| \pi_0 ) \\
    &= \sum_{k=1}^n y_k \langle \mathbf{x}_k, \mathbf{u} \rangle - \sum_{k=1}^n \mathbb{E}_{\bsigma \sim Q_{\mathbf{u},\mathbf{v}}} b( \langle \mathbf{x}_k, \bsigma \rangle ) + \frac{1}{2} \sum_{i=1}^p \log v_i - \sum_{i=1}^p \frac{v_i + u_i^2}{2} + \frac{p}{2}
\end{align*}
with respect to the variational parameters $\mathbf{u}$ and $\mathbf{v}$,
\begin{equation}
\label{eq:simulation_Gaussian_derivative}
\begin{aligned}
\frac{\partial \mathcal{M}}{\partial u_i} &= \sum_{k=1}^n y_k x_{ki} - \sum_{k=1}^n \mathbb{E}_{\bsigma \sim Q_{\mathbf{u},\mathbf{v}}} \left [   b( \langle \mathbf{x}_k, \bsigma \rangle ) \left ( - \frac{u_i - \sigma_i}{v_i} \right )  \right ] - u_i, \\
\frac{\partial \mathcal{M}}{\partial v_i} &= \sum_{k=1}^n \mathbb{E}_{\bsigma \sim Q_{\mathbf{u},\mathbf{v}}} \left [ b( \langle \mathbf{x}_k, \bsigma \rangle ) \left ( -\frac{1}{2 v_i} + \frac{(u_i - \sigma_i)^2 }{2v_i^2} \right )  \right ] + \frac{1}{2 v_i} - \frac{1}{2}.
\end{aligned}
\end{equation}
These derivatives can be approximately computed using a Monte Carlo strategy. Lastly, we use L-BFGS-B of \cite{byrd1995limited} to solve for 
$
    \sup_{ \mathbf{u} \in \mathbb{R}^p, \mathbf{v} \in \mathbb{R}_+^p  } \mathcal{M}(\mathbf{u}, \mathbf{v}).
$
We denote the resulting optimizer as $(\mathbf{u}_{\text{NMF}}, \mathbf{v}_{\text{NMF}})$ and the optimal value as $\mathcal{M}_{\text{NMF}}$, which serves as our NMF approximation to the log partition function $\log \mathcal{Z}_p$. 

\end{itemize} 

We summarize our findings in Figure~\ref{fig:Gaussian_prior_two_plots} and Table~\ref{table:Gaussian_prior_summary}. The NMF approximation to $\log \mathcal{Z}_p$ shows excellent agreement with the alternative estimates, thus illustrating the validity of this approximation. In addition, the MSE  for fitting the regression coefficients $\boldsymbol{\bbeta}$ obtained from the NMF scheme mirrors those obtained from the Monte Carlo/Jakkola-Jordan schemes. This further emphasizes the usefulness of this approximation. 

\begin{figure}[h]
\begin{minipage}{.4\linewidth}
    \centering
    \includegraphics[width=6cm]{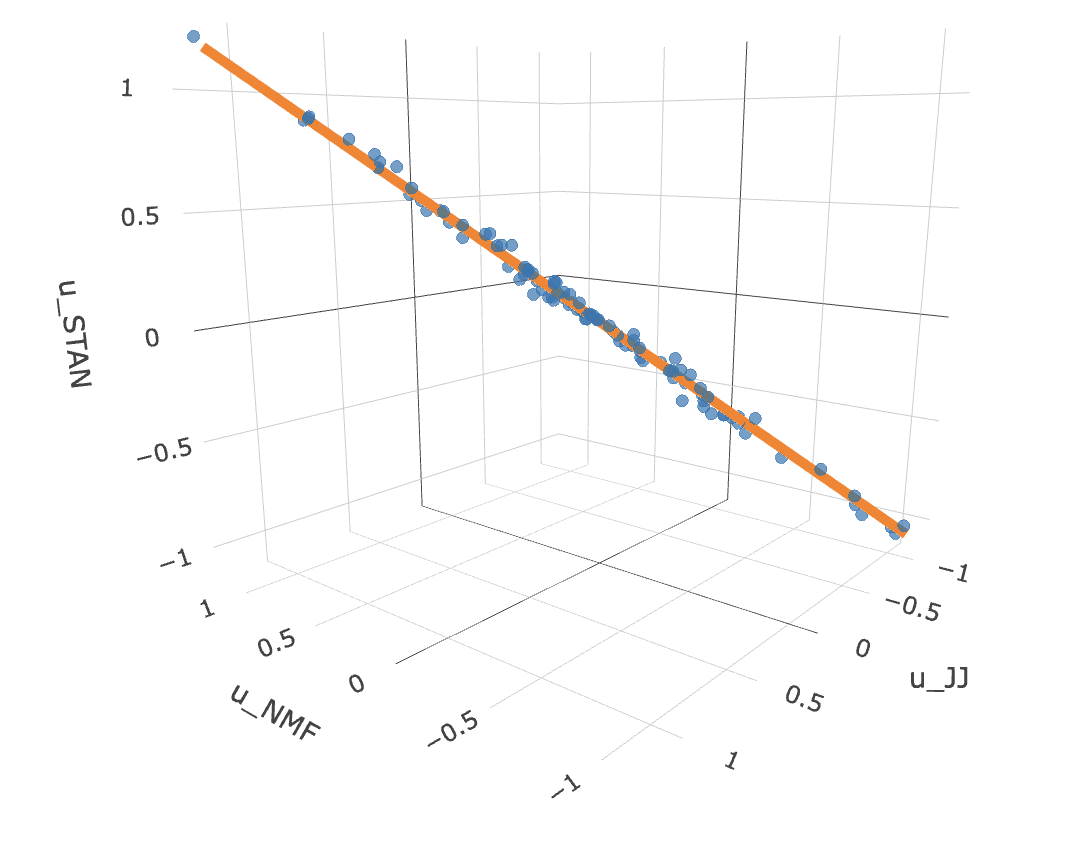}
\end{minipage}
\begin{minipage}{.6\linewidth}
    \centering
    \includegraphics[width=8cm]{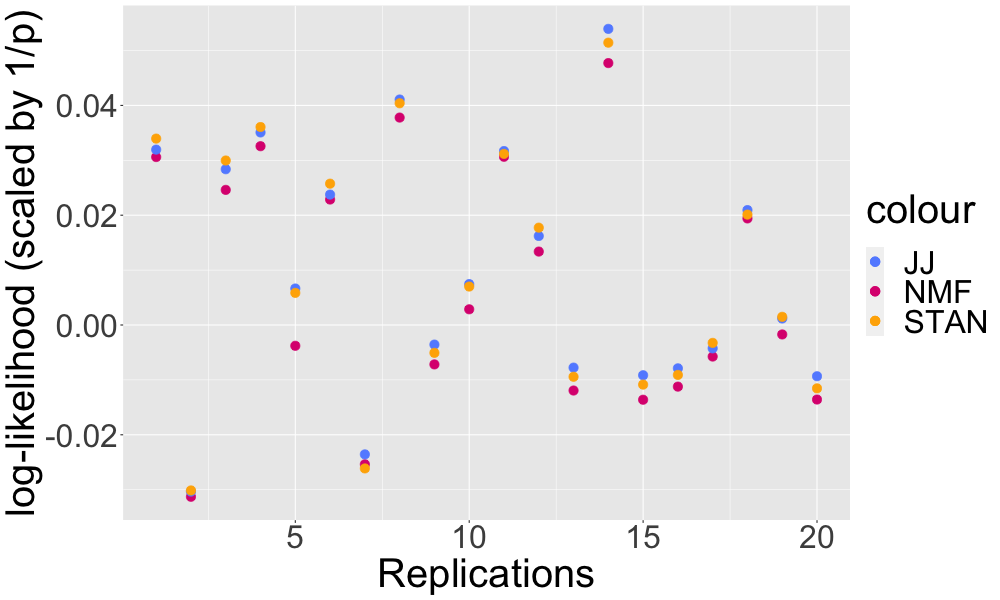}
\end{minipage}
\caption{The 3-D plot on the left visualizes the almost perfect alignment among $\mathbf{u}_{\text{NMF}}$, $\mathbf{u}_{\text{STAN}}$, and $\mathbf{u}_{\text{JJ}}$, when $n = 4000$ and $p = 100$. The right panel showcases the estimated log partition functions (y-axis) given by different methods, for 20 repeated experiments (x-axis). In particular, `STAN' stands for a sampling-based method called bridge sampling \cite{gronau2017bridgesampling}, and `JJ' refers to the widely celebrated tangent transform algorithm proposed by \cite{jaakkola2000bayesian}. Please note that for `JJ', the plotted value was not the evidence lower bound (ELBO). Instead, it was Monte Carlo evaluation of $\mathbb{E}_{\bsigma \sim Q^{\text{JJ}}} \left [ - H (\bsigma) \right ] -\operatorname{D}_{\text{KL}}(Q \| \pi_0 )  $, where $Q^{\text{JJ}}$ is the multivariate Gaussian distribution rendered by the Jaakkola and Jordan algorithm upon convergence. Here $n = 2000$ and $p = 50$.}
\label{fig:Gaussian_prior_two_plots}
\end{figure}

\begin{table}[h]
    \centering
    \begin{tabular}{@{} c|c|c|c @{}}
        $(p,n)$ & (50, 2000) & (100, 2000) & (100, 4000) \\
        \hline
        $ \left (\mathcal{M}_{\text{NMF}} - \log \mathcal{Z}_{\text{STAN}} \right ) / p$ & -0.0029 (0.0021) & -0.0047 (0.0013) & -0.0052 (0.0020)  \\
        \hline
        $ \left ( \mathcal{M}_{\text{NMF}} - \log \mathcal{Z}_{\text{JJ}}\right )/p $ & -0.0033 (0.0022) & -0.0045 (0.0014) & -0.0052 (0.0017) \\
        \hline
        STAN MSE & 0.719 (0.130) & 0.872 (0.114) & 0.808 (0.099)\\
        \hline
        JJ MSE & 0.719 (0.129) & 0.872 (0.114)  & 0.807 (0.098) \\
        \hline
        NMF MSE & 0.720 (0.126) & 0.870 (0.112)  & 0.813 (0.101) \\
    \end{tabular}
\caption{This table provides a compact summary of the comparisons between the three different methods, in terms of both log partition function approximation and mean square error with respect to the true hidden signal. The reported values are averages over 20 repeated experiments and the numbers in parentheses are standard deviations.}
\label{table:Gaussian_prior_summary}
\end{table}



\subsection{Logistic Regression with discrete priors}
\label{sec:discrete_algo} 

In this section, we focus on NMF Variational Inference for general GLMs with a non-gaussian prior on $[-1,1]$.   The Jakkola-Jordan algorithm crucially uses the gaussianity of the prior and does not generalize to this setting. Several instance specific algorithms have been introduced in the prior literature on this problem (see e.g, \cite{ray2020spike} and the references therein). In \cite{paisley2012variational,ranganath2014black}, the authors introduce a general idea, which they term ``black box" Variational Inference. The basic idea is the following: suppose one wishes to approximate the posterior distribution $\mu$ by minimizing $$\inf_{\upsilon \in \Upsilon} \mathrm{D}_{\mathrm{KL}}(Q_{\upsilon} \| \mu),$$ 
where $\Upsilon$ denotes an appropriate index set. In typical applications, $\Upsilon \subseteq \mathbb{R}^D$ for some $D \geq 1$.  Using \eqref{eq:ELBO1}, this is equivalent to maximizing the function 
\begin{align}
\upsilon \mapsto  \mathbb{E}_{\bbeta \sim Q_{\upsilon} } \left [ - H (\bbeta) \right ] -\operatorname{D}_{\text{KL}}(Q_{\upsilon} \| \pi_p ). \nonumber 
\end{align} 

To this end, \cite{paisley2012variational,ranganath2014black} implement a gradient descent approach, and estimate the gradients using Monte Carlo. We refer the interested reader to \cite{blei2017variational} and references therein for successful applications of this general strategy to diverse applications in statistics and machine learning. This general purpose strategy is extremely powerful if the variational family is a priori parametrized by some euclidean parameter $\upsilon \in \Upsilon$. 

This strategy is generally inapplicable for NMF Variational Inference i.e. if one wishes to approximate the posterior $\mu$ using the class of all product measures. This is specifically where our results can be useful.  Note that in this case, using \eqref{eq:ELBO1}, one wishes to optimize the map 
\begin{align}
Q \mapsto  \mathbb{E}_{\bbeta \sim Q } \left [ - H (\bbeta) \right ] -\operatorname{D}_{\text{KL}}(Q \| \pi_p ) \label{eq:dist_func}  
\end{align} 
over the set of product distributions absolutely continuous with respect to $\pi_p$. This is difficult to implement via an efficient algorithm. However, Theorem \ref{thm:alg_glm} establishes that to compute an approximate optimizer of \eqref{eq:dist_func}, we can restrict to sub-families $\{\mu_{\bbeta}: \bbeta \in [-1,1]^{p}\}$, $\{ \prod_{i=1}^{p} \pi_{(\gamma_j, d_j)}: \gamma_j \in \mathbb{R}\}$ or $\{Q_{\mathbf{u}} : \mathbf{u} \in [-1,1]^{p}\}$. Armed with this insight, one can re-implement the general strategy of ``black-box" Variational Inference described above. Thus our insights allow for a simple algorithmic implementation of NMF Variational Inference in general GLMs. We consider this to be an important practical takeaway of our results. 

We illustrate this general strategy in the context of a logistic regression model below. Our demonstration serves as a proof of concept --- in particular, we do not optimize our implementation for space/time complexity. To highlight our idea in a concrete setting, we assume that $ \pi_p:=\pi^{\otimes p}$ and $\pi$ is a probability distribution supported on $\{-1,0,1\}$. In particular, for $\xi_{-1}, \xi_0, \xi_1 \geq 0$ and $\xi_{-1} + \xi_0 + \xi_1 = 1$, let $\pi(\{-1\}) =  \xi_{-1}$, $\pi(\{0\})= \xi_0$, $\pi(\{1\})=\xi_1$, with $\xi_{-1} + \xi_0 + \xi_{1} = 1$. For our numerical experiments, we choose $\xi_{-1} = \xi_{1} = 0.2$, $\xi_{0} = 0.6$. Using Theorem~\ref{thm:alg_glm}, we use the Variational family 
%
%
%
\begin{align*}
    \Gamma = \Big \{ Q_{\mathbf{u}} := \prod_{i=1}^p \pi_{(h(u_i,d_i),d_i)} : \mathbf{u} \in (-1,1)^p \Big \},
\end{align*}
where $d_i := (\mathbf{X}^T \mathbf{X})_{ii}$. 
Define
\begin{align*}
    \mathcal{M}(\mathbf{u}) = \mathbb{E}_{\bsigma \sim Q_{\mathbf{u}}} \left [ - H (\bsigma) \right ] -\operatorname{D}_{\text{KL}}(Q_{\mathbf{u}} \| \pi_0 ).
\end{align*}

To optimize $\mathcal{M}(\cdot)$, we utilize the first order stationary point condition established in \eqref{eq:M_derivative}, \eqref{eq:v_defn} and construct an iterative scheme with variables $\{(\mathbf{u}_p^{(t)}, \mathbf{v}_p^{(t)}): t \geq 0\}$  initializing at $\mathbf{u}_p^{(0)} = \boldsymbol{0}$ as follows: 
\begin{align}
\label{eq:simulation_iterative_1}
v_{p,j}^{(t+1)}  &= \sum_{i=1}^{n} y_i x_{ij}  - \frac{\mathrm{Cov}_{\pi_{(h(u_{p,j}^{(t)},d_j),d_j)} }(f_j(\sigma_j; \mathbf{u}_p^{(t)}), \sigma_j)}{\ddot{c}_{\pi}(h(u_{p,j}^{(t)},d_j),d_j)} + b''(0) \frac{d_j}{2} \frac{\mathrm{Cov}_{\pi_{ (h(u_{p,j}^{(t)}, d_j), d_j) }} (\sigma_j^2,\sigma_j) }{\ddot{c}_{\pi}(h(u_{p,j}^{(t)}, d_j),d_j)},
\end{align}
\begin{align}
\label{eq:simulation_iterative_2}
\mathbf{u}_p^{(t+1)}  &= \dot{c}_\pi (\mathbf{v}_p^{(t + 1)}, \mathbf{d}),
\end{align}
where
\begin{align*}
f_j(\sigma_j ; \mathbf{u} ) = \mathbb{E}_{\pi_{(h(\mathbf{u},\mathbf{d}),\mathbf{d})} }\Big[\sum_{i=1}^{n} (b(\mathbf{x}_i^{\top} \mathbb{\sigma}) - b(\mathbf{x}_i^{\top} \mathbb{\sigma}_{0,j})) \Big| \sigma_j \Big].
\end{align*}
Note that for discrete distributions with finite support, the $\mathrm{Cov}$ functional above can be computed efficiently. The challenging step in this iteration is the computation of $f_j$ --- we use a Monte Carlo strategy to this end. We defer the analytic forms of  $\dot{c}_\pi$ and $\ddot{c}_{\pi}$ to the supplementary material. Upon convergence, we denote the final output as $\mathbf{u}_{\text{NMF}}$, which serves as the NMF approximation to the posterior mean vector, i.e., $\mathbb{E}_{\bbeta \sim \mu} [\mathbf{\bbeta}]$. In Figure~\ref{fig:discrete_prior}, we compare this $\mathbf{u}_{\text{NMF}}$ with the posterior mean approximated by a naive Gibbs sampler. 
\begin{figure}[h]
\begin{minipage}{.55\linewidth}
\includegraphics[width=7cm]{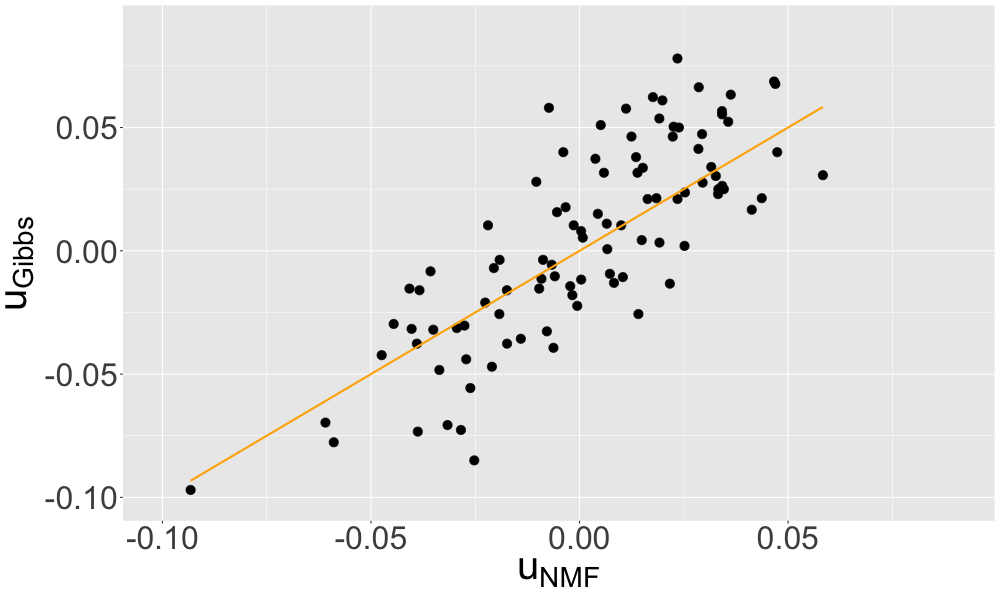}
\end{minipage}
\begin{minipage}{.4\linewidth}
    \begin{tabular}{@{} p{26mm}|c|c @{}}
         & Average & SD \\
        \hline
        NMF MSE & 0.3984219 & 0.005186545   \\
        \hline
        Gibbs MSE & 0.3971022 & 0.007278441   \\
    \end{tabular}
\end{minipage}
\caption{The left panel showcases a (typical) comparison between the approximations of the posterior mean vector given by our iterative scheme ($\mathbf{u}_{\text{NMF}}$ on the x-axis) and a naive Gibbs sampler ($\mathbf{u}_{\text{Gibbs}}$ on the y-axis) when $n = 2000$ and $p = 100$. These two estimators also have comparable mean square errors (MSEs) with respect to the true signal $\bbeta^{\star}$, as outlined in the table on the right (for $n = 1000$ and $p = 50$). The average and standard deviation were computed based on 10 repeated experiments. Entries of the design matrix $\mathbf{X}$ were sampled i.i.d. from $\mathcal{N}(0, 0.01 / n)$. Given the design, we sample $\mathbf{y}$ from a logistic regression model; the coefficients of the regression model are sampled i.i.d. from $\pi$, i.e., these two figure and table were generated assuming a well-specified setting.}
\label{fig:discrete_prior}
\end{figure}

\section{Discussion} 
\label{sec:discussion} 
In this section discuss our assumptions and collect some directions for future enquiry. 
\begin{itemize}
\item[(i)] \textbf{Design Assumptions:} Recall that most of our results are derived under Assumption~\ref{assump:design} on the design matrix $\mathbf{X}$. As remarked after Theorem~\ref{thm:covering}, there are well-known settings beyond Assumption~\ref{assump:design} where the NMF approximation is expected to fail. In particular, the Thouless-Anderson-Palmer (TAP) approximation is expected to yield a tight approximation to the log-partition function in the proportional asymptotic regime (i.e. $n \propto p$) with an i.i.d. gaussian design $\mathbf{X}$. Similarly, the Bethe approximation is expected to be tight under a proportional asymptotic regime (i.e. $n \propto p$) when the design matrix $\mathbf{X}$ corresponds to a sparse bi-partite graph with fixed average degrees. We believe rigorous investigations into these advanced mean-field approximations is an extremely fruitful direction for further enquiry.


\item[(ii)] \textbf{Unbounded Priors:} Throughout, we impose that the prior $\pi$ has a bounded support. Existing approaches to establish the NMF approximation typically require boundedness of the base distributions. The only counterexample is \cite{barbier2020strong}; however, they work under log-concavity assumptions on the base measures. From a practical perspective, it is important to extend these results to settings with unbounded priors (potentially with some additional tail-decay conditions). We believe this is an interesting direction for follw-up research, and leave this for future investigations.  
\item[(iii)] \textbf{The i.i.d. gaussian design case:} In \cite{barbier2019optimal}, the authors study Bayesian inference for GLMs under an i.i.d. gaussian design and a proportional asymptotic regime (i.e., $n \propto p$).  In this case, they view the model as a \emph{planted spin glass} and use ideas developed in this context (e.g. interpolation, cavity method) to analyze this model. Finally, this analysis assumes a well-specified setting, i.e. the response $\mathbf{y}$ is generated from the same model family; as remarked before, one has access to the powerful Nishimori identities in this setting. In sharp contrast, our results are valid in the regime $p = o(n)$, and thus complement these existing results. In addition, our results can accommodate correlations among the features, and are valid beyond the well-specified setting. While completing this manuscript, we came to know that in ongoing work \cite{saenz2024glm}, the authors are looking into Bayesian inference for misspecified GLMs under an i.i.d. Gaussian design and proportional asymptotic regime. These results will be in complementary scaling regimes, and thus are directly incomparable. In addition, their analysis builds on spin-glass based tools, and are thus distinct from the approach adopted in this paper. 
\end{itemize}

\section{Proofs} 
\label{sec:proofs} 

We establish Theorems~\ref{thm:covering} and \ref{thm:covering_suff} in Section~\ref{section:proofs12}. We prove Theorem~\ref{thm:alg_glm} in Section~\ref{sec:proofs3}. We establish Theorem~\ref{sec:proof4} in Section~\ref{thm:structure}.   

\subsection{Proof of Theorem \ref{thm:covering} and Theorem \ref{thm:covering_suff} } 
\label{section:proofs12}

\begin{proof}[Proof of Theorem \ref{thm:covering}]
    We prove this result using the machinery derived in \cite{austin2019structure}. Consider the Hamiltonian $H(\boldsymbol{\beta}) = \sum_{i = 1}^{n}   \left (  - y_i \langle \mathbf{x}_i, \bbeta \rangle + b(\langle \mathbf{x}_i, \bbeta \rangle) \right ) := \sum_{i=1}^{n} g_i(\langle \mathbf{x}_i, \bbeta \rangle)$. 
    
    We first define the discrete gradient function: fix $\boldsymbol{\beta}, \boldsymbol{\sigma} \in [-1,1]^p$. The discrete gradient of $H(\cdot)$ at $\boldsymbol{\beta}$, evaluated at $\boldsymbol{\sigma}$ is defined as 
    \begin{align}
        \nabla H(\boldsymbol{\sigma}; \boldsymbol{\beta}) = \sum_{j=1}^{p} \sum_{i=1}^{n} \left( g_i(\langle \mathbf{x}_i , \boldsymbol{\beta}_{\sigma_j,j} \rangle) -  g_i(\langle \mathbf{x}_i , \boldsymbol{\beta}_{0,j} \rangle) \right), \nonumber 
    \end{align}
    where $\boldsymbol{\beta}_{y,j} $ is the vector where the $j^{th}$ entry is $y$, and the remaining entries are the same as $\boldsymbol{\beta}$. According to the formalism introduced in Austin \cite{austin2019structure}, if these discrete gradient functions can be efficiently covered by a net of size $e^{o(p)}$, then NMF is correct to leading order. We will show that under the theses of this theorem, there exists an efficient covering of the discrete gradient functions. 

    To this end, we use Taylor expansion on the function $f_i:[0,1] \to \mathbb{R}$, $f_i(t) = g_i(\langle \mathbf{x}_i, \bbeta_{0,j} \rangle + t x_{ij} \sigma_j)$, to conclude that  
    \begin{align}
        g_i(\langle \mathbf{x}_i , \boldsymbol{\beta}_{\sigma_j,j} \rangle ) = g_i (\langle \mathbf{x}_i, \boldsymbol{\beta}_{0,j} \rangle) + x_{ij} \sigma_j g'_i(\langle \mathbf{x}_i, \boldsymbol{\beta}_{0,j}\rangle) + x_{ij}^2\sigma_j^2 \int_0^1 (1-t) g''_i(\langle \mathbf{x}_i, \boldsymbol{\beta}_{0,j} + t x_{ij} \sigma_j \rangle ) dt. \nonumber  
    \end{align}

    We note that 
    \begin{align}
        \Big| &\sum_{j=1}^{p} \sum_{i=1}^{n} \left( g_i(\langle \mathbf{x}_i , \boldsymbol{\beta}_{\sigma_j,j} \rangle) -  g_i(\langle \mathbf{x}_i , \boldsymbol{\beta}_{0,j} \rangle) \right)   
        - \sum_{j=1}^{p} \sum_{i=1}^{n} \Big(  x_{ij} \sigma_j g_i'(\langle \mathbf{x}_i, \boldsymbol{\beta}_{0,j}\rangle) + \frac{x_{ij}^2\sigma_j^2}{2} g_i''(\langle \mathbf{x}_i, \boldsymbol{\beta}_{0,j} \rangle ) \Big) \Big| \nonumber \\
        &\leq \sum_{i=1}^{n} \sum_{j=1}^{p} x_{ij}^2 \int_0^1 | b''(\langle \mathbf{x}_i, \bbeta_{0,j} \rangle + t x_{ij} \sigma_j) - b''(\langle \mathbf{x}_i, \bbeta_{0,j} \rangle ) | dt, \nonumber  
    \end{align}
    where we used the fact that $\int_0^1 (1-t) \mathrm{d} t = 1/2$ and $|\sigma_j| \le 1$. Using Assumption \ref{assump:cont}, for any $\varepsilon>0$, there exists $\delta>0$ such that 
   \begin{align}
   \Big| &\sum_{j=1}^{p} \sum_{i=1}^{n} \left( g_i(\langle \mathbf{x}_i , \boldsymbol{\beta}_{\sigma_j,j} \rangle) -  g_i(\langle \mathbf{x}_i , \boldsymbol{\beta}_{0,j} \rangle) \right)   
        - \sum_{j=1}^{p} \sum_{i=1}^{n} \Big(  x_{ij} \sigma_j g_i'(\langle \mathbf{x}_i, \boldsymbol{\beta}_{0,j}\rangle) + \frac{x_{ij}^2\sigma_j^2}{2} g_i''(\langle \mathbf{x}_i, \boldsymbol{\beta}_{0,j} \rangle ) \Big) \Big| \nonumber \\ 
        &\lesssim \varepsilon \sum_{i=1}^{n} \sum_{j=1}^{p} x_{ij}^2 + \sum_{i=1}^n \sum_{j=1}^{p} x_{ij}^2 \mathbf{1}(|x_{ij}|> \delta) = o(p). \label{eq:sec_order} 
   \end{align}  
   
   The last equality follows from first sending $p \to \infty$ and noting that $\varepsilon>0$ is arbitrary. We will use similar equalities in our subsequent computations, and will omit the corresponding justification whenever there is no scope for confusion. 


Thus at a $o(p)$ cost, the discrete gradient can be approximated by an intermediate function 
\begin{align}
    G_1(\boldsymbol{\sigma}; \boldsymbol{\beta}) = \sum_{j=1}^{p} \sum_{i=1}^{n} \Big(  x_{ij} \sigma_j g_i'(\langle \mathbf{x}_i, \boldsymbol{\beta}_{0,j}\rangle) + \frac{x_{ij}^2\sigma_j^2}{2} g_i''(\langle \mathbf{x}_i, \boldsymbol{\beta}_{0,j} \rangle ) \Big). \nonumber 
\end{align}
Again, using Taylor expansion, 
\begin{align}
    &\sum_{j=1}^{p} \sum_{i=1}^{n} x_{ij} \sigma_j g'_i(\langle \mathbf{x}_i, \boldsymbol{\beta}_{0,j} \rangle)  
    = \sum_{j=1}^{p} \sum_{i=1}^{n} x_{ij} \sigma_j g_i'(\langle \mathbf{x}_i, \boldsymbol{\beta} \rangle)  - \sum_{j=1}^{p} \sum_{i=1}^{n} x_{ij}^2 \sigma_j \beta_j \int_0^1 g''_i(\langle \mathbf{x}_i, \boldsymbol{\beta}\rangle - t x_{ij} \beta_j ) dt.  \nonumber 
 \end{align}
 This implies 
 \begin{align}
 \Big| \sum_{j=1}^{p} \sum_{i=1}^{n} x_{ij} \sigma_j g'_i(\langle \mathbf{x}_i, \boldsymbol{\beta}_{0,j} \rangle)  - \Big( \sum_{j=1}^{p} \sum_{i=1}^{n} x_{ij} \sigma_j g_i'(\langle \mathbf{x}_i, \boldsymbol{\beta} \rangle) - \sum_{j=1}^{p} \sum_{i=1}^{n} x_{ij}^2 \sigma_j \beta_j  g''_i(\langle \mathbf{x}_i, \boldsymbol{\beta}\rangle)  \Big)  \Big| = o(p), \nonumber 
 \end{align} 
 where the last equality follows by the same argument as \eqref{eq:sec_order}.  
 Finally, applying the same argument as \eqref{eq:sec_order} to second term in $G_1$, we have, 
 \begin{align}
     \sum_{j=1}^{p} \sum_{i=1}^{n} x_{ij}^2 \sigma_j^2 g_i''(\langle x_i , \boldsymbol{\beta}_{0,j} \rangle ) = \sum_{j=1}^{p} \sum_{i=1}^{n} x_{ij}^2 \sigma_j^2 g_i''(\langle \mathbf{x}_i, \boldsymbol{\beta}\rangle ) + o(p) . \nonumber 
 \end{align}
 In turn, this yields the following approximation to the discrete gradient function 
 \begin{align}
     G_2(\boldsymbol{\sigma}; \boldsymbol{\beta}) 
     = \sum_{j=1}^{p} \sum_{i=1}^{n} x_{ij} \sigma_j g_i'(\langle \mathbf{x}_i, \boldsymbol{\beta} \rangle) - \sum_{j=1}^{p} \sum_{i=1}^{n} x_{ij}^2 \sigma_j \beta_j g_i''(\langle \mathbf{x}_i, \boldsymbol{\beta} \rangle) + \frac{1}{2} \sum_{j=1}^{p} \sum_{i=1}^{n} x_{ij}^2 \sigma_j^2 g_i''(\langle \mathbf{x}_i , \boldsymbol{\beta} \rangle ). \nonumber 
 \end{align}
Observe that the last term can be approximated as follows: 

\begin{align}
    \sum_{j=1}^{p} \sum_{i=1}^{n} x_{ij}^2 \sigma_j^2 g_i''(\langle \mathbf{x}_i, \boldsymbol{\beta}\rangle) = \sum_{j=1}^{p} \sum_{i=1}^{n} x_{ij}^2 \sigma_j^2 g_i''(0) + \sum_{j=1}^{p} \sum_{i=1}^{n} x_{ij}^2 \sigma_j^2  (g_i''(\langle \mathbf{x}_i, \boldsymbol{\beta}\rangle) - g_i''(0)). \nonumber 
\end{align}

Now, define $\mathcal{S}(\boldsymbol{\beta}, \varepsilon) :=\{ i \in [n]: |\langle \mathbf{x}_i, \boldsymbol{\beta} \rangle | > \varepsilon\}$. By Markov inequality, 
\begin{align}
    |\mathcal{S}(\boldsymbol{\beta},\varepsilon) | \leq \frac{1}{\varepsilon^2}\sum_{i \in \mathcal{S}(\boldsymbol{\beta}, \varepsilon)} \langle \mathbf{x}_i, \boldsymbol{\beta} \rangle^2 \leq \frac{1}{\varepsilon^2}\sum_{i \in [n]} \langle \mathbf{x}_i, \boldsymbol{\beta} \rangle^2  \leq \frac{\| \mathbf{X}^{\top} \mathbf{X}\|_2}{\varepsilon^2} p. \nonumber 
\end{align}

We note that 
\begin{align}
    &\sum_{j=1}^{p} \sum_{i=1}^{n} x_{ij}^2 \sigma_j^2 (g_i''(\langle \mathbf{x}_i, \boldsymbol{\beta} \rangle) - g_i''(0)) \nonumber \\
    &= \sum_{i \in \mathcal{S}(\boldsymbol{\beta},\varepsilon)^c} \sum_{j=1}^{p} x_{ij}^2 \sigma_j^2 (g_i''(\langle \mathbf{x}_i, \boldsymbol{\beta} \rangle) - g_i''(0)) + \sum_{i \in \mathcal{S}(\boldsymbol{\beta},\varepsilon)} \sum_{j=1}^{p} x_{ij}^2 \sigma_j^2 (g_i''(\langle \mathbf{x}_i, \boldsymbol{\beta} \rangle) - g_i''(0)). \nonumber 
\end{align}

Next, we observe that 
\begin{align}
    &\sum_{i \in \mathcal{S}(\boldsymbol{\beta},\varepsilon)^c} \sum_{j=1}^{p} x_{ij}^2 \sigma_j^2 (g_i''(\langle \mathbf{x}_i, \boldsymbol{\beta} \rangle) - g_i''(0)) 
    \leq \sup_{|x| \leq \varepsilon} |b''(x) - b''(0)| \mathrm{Tr}(\mathbf{X}^{\top} \mathbf{X}) + \sup_{|S| \leq Cp} \sum_{i \in S} \sum_{j=1}^{p} x_{ij}^2 = o(p), \nonumber 
 \end{align}
Note that $C>0$ depends on $\varepsilon>0$ and $C \to \infty$ as $\varepsilon \to 0$. 
where the last inequality uses the observation $g_i''(x) = b''(x)$ and the fact that $\sum_{i,j} x_{ij}^2 = \mathrm{Tr}(\mathbf{X}^{\top} \mathbf{X}) $. 
 In turn, we obtain the following approximation of the discrete gradient function.  
 \begin{align}
     G_3(\boldsymbol{\sigma}; \boldsymbol{\beta}) &= \sum_{j=1}^{p} \sum_{i=1}^{n} x_{ij} \sigma_j g_i'(\langle \mathbf{x}_i, \boldsymbol{\beta} \rangle) - \sum_{j=1}^{p} \sum_{i=1}^{n} x_{ij}^2 \sigma_j \beta_j g_i''(\langle \mathbf{x}_{i}, \boldsymbol{\beta}\rangle). \nonumber \\
     &= - \bsigma^{\top}\mathbf{X}^\top \mathbf{y} +  \boldsymbol{\sigma}^{\top} \mathbf{X}^{\top}b'(\mathbf{X} \boldsymbol{\beta}) - \sum_{j=1}^{p} \sum_{i=1}^{n} x_{ij}^2 \sigma_j \beta_j b''(\langle \mathbf{x}_{i}, \boldsymbol{\beta}\rangle). \label{eq:G3_defn} 
 \end{align}
 We note that we have dropped a term $\sum_{j=1}^{p} \sum_{i=1}^{n} x_{ij}^2 \sigma_j^2 g_i''(0)$ in the display above as it does not depend on $\bbeta$---it does not affect the subsequent net construction. 

 Consider now the function $\psi: [0,1] \to \mathbb{R}^n$, $t \mapsto \sum_{j=1}^{p} \sum_{i=1}^{n} x_{ij} \sigma_j g_i'(\langle \mathbf{x}_i, t \boldsymbol{\beta} \rangle)$. By Taylor expansion, we have, 
 \begin{align}
     \psi(1) = \psi(0) + \psi'(t^*) \nonumber 
 \end{align}
 for some $t_* \in [0,1]$. By direct computation, 
 \begin{align}
     \psi'(t^*) = \sum_{j=1}^{p} \sum_{i=1}^{n} x_{ij} \sigma_j b''(\langle \mathbf{x}_i , t^* \boldsymbol{\beta} \rangle) \langle \mathbf{x}_i , \boldsymbol{\beta} \rangle.  \nonumber 
 \end{align}

 Finally we again have
 \begin{align}
     &\Big|\sum_{j=1}^{p} \sum_{i=1}^{n} x_{ij}^2 \sigma_j \beta_j b''(\langle \mathbf{x}_i, \boldsymbol{\beta}\rangle) - \sum_{j=1}^{p} \sum_{i=1}^{n} x_{ij}^2 \sigma_j \beta_j b''(\langle \mathbf{x}_i, t_* \boldsymbol{\beta}\rangle ) \Big| \nonumber \\
     &\leq \sup_{|x|, |y| \leq \varepsilon} |b''(x) - b''(y)| \sum_{i=1}^{n}\sum_{j=1}^{p} x_{ij}^2 + \|b''\|_{\infty} \sum_{i \in \mathcal{S}(\boldsymbol{\beta},\varepsilon) \cup \mathcal{S}(t_*\boldsymbol{\beta}, \varepsilon)} \sum_{j=1}^{p}  x_{ij}^2 \nonumber \\ 
     &\leq \sup_{|x|, |y| \leq \varepsilon} |b''(x) - b''(y)| \sum_{i=1}^{n}\sum_{j=1}^{p} x_{ij}^2 + \|b''\|_{\infty}  \sup_{|S|\leq Cp } \sum_{i \in S} \sum_{j=1}^{p} x_{ij}^2 = o(p). 
 \end{align}

This implies we have to cover the set of functions 
\begin{align} 
\Big\{ \sigma \mapsto \sum_{j=1}^{p} \sum_{i=1}^{n} x_{ij} \sigma_j b''(\langle \mathbf{x}_i, t_* \boldsymbol{\beta}\rangle) \langle \mathbf{x}_i, \boldsymbol{\beta} \rangle - \sum_{j=1}^{p} \sum_{i=1}^{n} x_{ij}^2 \sigma_j \beta_j b''(\langle \mathbf{x}_i, t_* \boldsymbol{\beta} \rangle): \boldsymbol{\beta} \in [-1,1]^p  \Big\}. \label{eq:int_func} 
\end{align}  
 
 Note that 
 \begin{align}
 &\sum_{j=1}^{p} \sum_{i=1}^{n} x_{ij} \sigma_j b''(\langle \mathbf{x}_i, t_* \boldsymbol{\beta}\rangle) \langle \mathbf{x}_i, \boldsymbol{\beta} \rangle - \sum_{j=1}^{p} \sum_{i=1}^{n} x_{ij}^2 \sigma_j \beta_j b''(\langle \mathbf{x}_i, t_* \boldsymbol{\beta} \rangle) \nonumber = \bsigma^{\top} \mathbf{A}_{t_* \bbeta} \bbeta. \nonumber 
 \end{align} 
 
 Note that as long as the set of vectors $\{ \mathbf{A}_{\bbeta} \bgamma: \bbeta, \bgamma \in [-1,1]^p\}$ has an efficient cover in $L^1$ norm, the linear functions can be covered in the $L^\infty$ norm.

\end{proof}

\begin{proof}[Proof of Theorem \ref{thm:covering_suff}] 
Using Theorem \ref{thm:covering}, it suffices to cover $\{\mathbf{A}_{\bbeta} \bgamma: \bbeta, \bgamma \in [-1,1]^{p} \}$. For any $\varepsilon>0$, without loss of generality, let $\{\mathbf{M}^{\varepsilon}_1, \cdots, \mathbf{M}^{\varepsilon}_{N(p,\varepsilon)}\}$ be the smallest $\varepsilon$-net in $L^2$ operator norm for the collection $\{\mathbf{A}_{\bbeta}: \bbeta \in [-1,1]^p\}$. In particular, this implies that for any $1\leq j \leq N(p, \varepsilon)$, there exists $\bbeta \in [-1,1]^p$ such that $\|\mathbf{A}_{\bbeta} - \mathbf{M}^{\varepsilon}_j \|_2 < \varepsilon$---if not, we can reduce the size of the net by dropping $\mathbf{M}^{\varepsilon}_j$. Define $\mathcal{C}(\varepsilon) = \{ \mathbf{M}^\varepsilon_j \bgamma: 1\leq j \leq N(p,\varepsilon) , \bgamma \in [-1,1]^{p}\}$. Observe that for any $\bbeta \in [-1,1]^p$, there exists $1\leq j \leq N(p, \varepsilon)$ such that 
for any $\bgamma \in [-1,1]^{p}$, 
\begin{align}
\| \mathbf{A}_{\bbeta} \bgamma - \mathbf{M}^{\varepsilon}_j \bgamma \|_1 \leq p \| \mathbf{A}_{\bbeta} - \mathbf{M}^{\varepsilon}_j \|_2 \leq p \varepsilon, \nonumber 
\end{align} 
where the first inequality follows from Cauchy-Schwarz inequality. Thus it suffices to cover the set $\mathcal{C}(\varepsilon)$ in $L^1$ norm. Note that for any $1\leq j \leq N(p,\varepsilon)$, there exists $\bbeta \in [-1,1]^p$ 
\begin{align}
\mathrm{Tr}((\mathbf{M}^{\varepsilon}_j)^2) = \| \mathbf{M}^{\varepsilon}_j \|_F^2 \leq 2 \| \mathbf{A}_{\bbeta} \|_F^2 + 2 p \varepsilon^2 \leq 4 p \varepsilon^2.  \nonumber 
\end{align} 
For each $1\leq j \leq N(p,\varepsilon)$, we use \cite[Lemma 3.4]{basak2017universality} to construct a cover of the set $\{ \mathbf{M}^\varepsilon_j \bgamma: \bgamma \in [-1,1]^p\}$. Denote this set as $\mathcal{C}_j(\varepsilon)$ and denote its size as $s_j(\varepsilon)$. In particular, $\cup_{j=1}^{N(p,\varepsilon)} \mathcal{C}_j$ is a cover of $\mathcal{C}(\varepsilon)$ with size $\sum_{j=1}^{N(p,\varepsilon)} s_j (\varepsilon)$. Finally we note that 
\begin{align}
\lim_{\varepsilon \to 0} \lim_{p \to \infty} \frac{1}{p} \log  \sum_{j=1}^{N(p,\varepsilon)} s_j (\varepsilon) =0. \nonumber 
\end{align}
This completes the proof.  
\end{proof}

\begin{proof}[Proof of Corollary \ref{cor:applications}]
We start with the proof of part (i). Using Theorem \ref{thm:covering_suff}, we have to cover the set of vectors $\{ \mathbf{A}_{\bbeta} \bgamma: \bbeta, \bgamma \in [-1,1]^p\}$ in $L^1$ norm. Recalling $\mathbf{A}_{\bbeta} = \mathbf{X}^{\top} \mathbf{D}_{\bbeta} \mathbf{X} - \mathrm{diag}(\mathbf{X}^{\top} \mathbf{D}_{\bbeta} \mathbf{X})$, we note that 
\begin{align}
\mathbf{X}^{\top} \mathbf{D}_{\bbeta} \mathbf{X} = \sum_{i=1}^{n} b''( \langle \mathbf{x}_i, \bbeta \rangle ) \mathbf{x}_i \mathbf{x}_i^{\top}. \nonumber 
\end{align} 
Setting $\mathcal{S}(\bbeta, \varepsilon) = \{ i \in [n] : |\langle \mathbf{x}_i, \bbeta \rangle| > \varepsilon \}$, we have
\begin{align}
\mathbf{X}^{\top} \mathbf{D}_{\bbeta} \mathbf{X} = \sum_{i \in \mathcal{S}(\bbeta, \varepsilon)^c} b''( \langle \mathbf{x}_i, \bbeta \rangle ) \mathbf{x}_i \mathbf{x}_i^{\top} + \sum_{i \in \mathcal{S}(\bbeta, \varepsilon)} b''( \langle \mathbf{x}_i, \bbeta \rangle ) \mathbf{x}_i \mathbf{x}_i^{\top}. \nonumber 
\end{align}
Using the definition of $\mathcal{S}(\bbeta, \varepsilon) = \{ i \in [n] : |\langle \mathbf{x}_i, \bbeta \rangle| > \varepsilon \}$ and Lemma \ref{lem:perturbation}, we have
\begin{align}
\| \mathbf{X}^{\top} \mathbf{D}_{\bbeta} \mathbf{X} - b''(0) \mathbf{X}^{\top} \mathbf{X} \|_{\mathrm{op}} \leq \delta(\varepsilon) \| \mathbf{X}^{\top} \mathbf{X} \|_{\mathrm{op}} + \| b'' \|_{\infty} \sup_{|S| \leq Cp} \Big\| \sum_{i \in S} \mathbf{x}_i \mathbf{x}_i^{\top}  \Big\|_{\mathrm{op}} \lesssim \delta(\varepsilon) + o(1), \nonumber  
\end{align} 
where $\delta(\varepsilon) = \max_{x \in [-\varepsilon,\varepsilon]} | b''(x) - b''(0)| $.
In addition, Lemma \ref{lemma:diag_dominance} implies 
\begin{align}
\|\mathrm{diag}(\mathbf{X}^{\top} \mathbf{D}_{\bbeta} \mathbf{X}) - b''(0)  \mathrm{diag}(\mathbf{X}^{\top} \mathbf{X}) \|_{\mathrm{op}} \leq \| \mathbf{X}^{\top} \mathbf{D}_{\bbeta} \mathbf{X} - b''(0) \mathbf{X}^{\top} \mathbf{X} \|_{\mathrm{op}} = o(1).  \nonumber 
\end{align} 
Combining, we immediately have, $\sup_{\bbeta \in [-1,1]^p} \| \mathbf{A}_{\bbeta} - b''(0) \mathbf{A} \|_{\mathrm{op}} = o(1)$. The set $\{ \mathbf{A}_{\bbeta} : \bbeta \in [-1,1]^{p} \}$ can be covered by the singleton set $\{ b''(0) \mathbf{A} \}$. Further, if $\mathrm{Tr}(\mathbf{A}^2) = o(p)$, $\sup_{\bbeta \in [-1,1]^{p} } \mathrm{Tr}(\mathbf{A}_{\bbeta}^2) = o(p)$ --- this follows from eigenvalue interlacing. The proof is now complete by invoking Theorem \ref{thm:covering_suff}.

 We now turn to the proof of part (ii). We first verify that Assumption~\ref{assump:design} holds for designs with i.i.d. gaussian rows. Specifically, $n \gg p$ implies that $\| \mathbf{X}^{\top} \mathbf{X} - \boldsymbol{\Sigma}_p \|_2 = o(1)$ \cite{vershynin2018high}. Similarly, setting $x_{ij} = z_{ij}/\sqrt{n}$, we obtain 
 \begin{align*}
 \mathbb{E}[\sum_{i=1}^{n} \sum_{j=1}^{p} x_{ij}^2 \mathbf{1}(|x_{ij}| > \delta)] = \sum_{j=1}^{p} \mathbb{E}[z_{1j}^2 \mathbf{1}(|z_{ij}| > \delta \sqrt{n})] = o(p),
 \end{align*}
 where the last equality follows from $\| \boldsymbol{\Sigma}_p \|_2 \leq C$. Finally, for $S \subset [n]$, by Bernstein's inequality
 \begin{align*}
 \mathbb{P}\Big[  \Big| \sum_{i \in S} \| (\mathbf{x}_i\|_2^2 - \mathbb{E}[\|\mathbf{x}_i \|_2^2] )   \Big| > t \Big] \leq 2 \exp(- c \frac{t}{K}),  
 \end{align*} 
 where $c>0$ is a universal constant, and $K$ is the 1-Orlicz norm of $\| \mathbf{x}_i \|_2^2$. Now, we note that under the given assumptions on $\Sigma_p$, $K = O(p/n)$. We can now verify condition (iii) in Assumption~\ref{assump:design} by union bound, noting that $\log {n \choose Cp} = n \log (p/n)$.

Now, using the proof of \cite[Corollary 2]{mukherjee2021variational}, we know that $\mathrm{Tr}(\mathbf{A}^2) =o(p)$ with high probability in this setting.  Let $\mathbf{z}_i \sim \mathcal{N}(0,\mathbf{\Sigma}_p)$. Then there exists $C_1>0$ (depending on $\mathbf{\Sigma}_p$) and $C'>0$ depending on $C_1$ such that  \cite{vershynin2018high}
 \begin{align}
 \mathbb{P}\Big( \sigma_{\max}(\mathbf{Z}_S) > \sqrt{|S|} +  C_1 \sqrt{p} + t \Big) \leq 2 \exp(- C' t^2). \nonumber 
 \end{align} 
We choose $t = t_0 \sqrt{n}$ so that 
 \begin{align}
 \mathbb{P}\Big( \| \sum_{i \in S} \mathbf{x}_i \mathbf{x}_i^{\top} \|_{\mathrm{op}} > 2 t_0^2   \Big) \leq 2 \exp(- C' n t_0^2). \nonumber 
 \end{align} 
The proof of (ii) now follows by union bound. 

Finally, we move to the proof of part (iii). Assume that $p/n \to \kappa \in (0,\infty)$. For any $\bbeta \in [-1,1]^p$, we have, 
\begin{align}
\mathrm{Tr}(\mathbf{A}_{\bbeta}^2)  = \| \mathbf{A}_{\bbeta} \|_F^2 \leq 2 \| \mathbf{X}^{\top} \mathbf{D}_{\bbeta} \mathbf{X} \|_F^2 \leq 2 \lambda_{\max} (\mathbf{X}^{\top} \mathbf{D}_{\bbeta} \mathbf{X}) \mathrm{Tr}(\mathbf{X}^{\top} \mathbf{D}_{\bbeta} \mathbf{X}). \nonumber 
\end{align} 
In addition, note that $\mathrm{Tr}(\mathbf{X}^{\top} \mathbf{D}_{\bbeta} \mathbf{X}) \leq \| b'' \|_{\infty} \mathrm{Tr}(\mathbf{X}^{\top} \mathbf{X}) = o(p)$ using \eqref{eq:norm}. Similarly, we have that $\lambda_{\max} (\mathbf{X}^{\top} \mathbf{D}_{\bbeta} \mathbf{X}) \leq \| b'' \|_{\infty} \lambda_{\max}(\mathbf{X}^{\top} \mathbf{X}) = O(1)$. Thus we have $\sup_{\bbeta \in [-1,1]^{p}} \mathrm{Tr}(\mathbf{A}_{\bbeta}^2) = o(p)$. Now we construct a sub-exponential size net (in operator norm) for the set of matrices $\{\mathbf{A}_{\bbeta} : \bbeta \in [-1,1]^{p} \}$. 
Using Lemma \ref{lemma:diag_dominance}, it suffices to cover the set $\{ \mathbf{X}^{\top} \mathbf{D}_{\bbeta} \mathbf{X} : \bbeta \in [-1,1]^{p} \}$. To this end, we start with an $L^2$-net $\mathcal{N}_1(p, \varepsilon)$ of $\{\mathbf{X} \bbeta : \bbeta \in [-1,1]^{p}\}$. As $p/n \to \kappa$, there exists an efficient net such that $\lim_{\varepsilon \to 0} \lim_{p \to \infty} \frac{1}{p} \log |\mathcal{N}_1(p,\varepsilon)| = 0$. Consider first the set 
$\{ \sum_{i=1}^{n} b''(v_i) \mathbf{x}_i \mathbf{x}_i^{\top} : \mathbf{v} \in \mathcal{N}_1(p,\varepsilon)\}$. For $\bbeta \in [-1,1]^p$, assume that $\| \mathbf{X} \bbeta - \mathbf{v}\|_2^2 \leq p \varepsilon$ for $\mathbf{v} \in \mathcal{N}_1(p,\varepsilon)$. For any $\delta>0$, we have, 
\begin{align}
|\{ i: |\langle \mathbf{x}_i, \bbeta \rangle - v_i | > \delta\}| \leq \frac{p \varepsilon }{\delta^2}. \nonumber 
\end{align} 
Define $\mathcal{S}(\bbeta, \mathbf{v}, \delta) = \{ i : |\langle \mathbf{x}_i, \bbeta \rangle - v_i | \leq  \delta\}$. Define $\mathcal{C}(\varepsilon)$ be an $\varepsilon$-net of $[-2\|b''\|_{\infty}, 2\|b''\|_{\infty}]$, and set $\mathcal{N}_2(p, \varepsilon ) = \{ \sum_{i \in S} w_i \mathbf{x}_i \mathbf{x}_i^{\top}: S \subset [n], |S| \leq p \varepsilon , \mathbf{w} \in \mathcal{C}_{\varepsilon}^{|S|} \}$. We note that 
\begin{align}
|\mathcal{N}_2(p,\varepsilon) | \leq {p \choose p \varepsilon} |\mathcal{C}_{\varepsilon} |^{p \varepsilon}. \nonumber 
\end{align} 
Noting that $|\mathcal{C}_{\varepsilon}| \leq 1/\varepsilon$, we have that $\lim_{\varepsilon \to 0} \lim_{p \to \infty} \frac{1}{p} \log |\mathcal{N}_2(p,\varepsilon)| =0$. Finally, we define the net 
\begin{align}
\mathcal{N}(p,\varepsilon) = \{ \mathbf{M}_1 + \mathbf{M}_2: \mathbf{M}_1 \in \mathcal{N}_1(p, \varepsilon), \mathbf{M}_2 \in \mathcal{N}_2(p, \varepsilon)\}.  \nonumber 
\end{align}
We claim that this net suffices for our purposes. To see this, note that 
\begin{align}
\sum_{i=1}^{n} b''(\langle \mathbf{x}_i, \bbeta \rangle ) \mathbf{x}_i \mathbf{x}_i^{\top}  = \sum_{i \in \mathcal{S}(\bbeta, \mathbf{v} , \varepsilon^{1/3}) } b''(\langle \mathbf{x}_i, \bbeta \rangle ) \mathbf{x}_i \mathbf{x}_i^{\top}  + \sum_{i \in \mathcal{S}(\bbeta, \mathbf{v} , \varepsilon^{1/3})^c }b''(\langle \mathbf{x}_i, \bbeta \rangle ) \mathbf{x}_i \mathbf{x}_i^{\top}. \nonumber 
\end{align}
Using Lemma \ref{lemma:diag_dominance}, 
\begin{align}
\Big\|   \sum_{i \in \mathcal{S}(\bbeta, \mathbf{v} , \varepsilon^{1/3}) } b''(\langle \mathbf{x}_i, \bbeta \rangle ) \mathbf{x}_i \mathbf{x}_i^{\top}  - \sum_{i \in \mathcal{S}(\bbeta, \mathbf{v} , \varepsilon^{1/3}) } b''(v_i ) \mathbf{x}_i \mathbf{x}_i^{\top} \Big\|_{\mathrm{op}} \lesssim \delta(\varepsilon). \nonumber 
\end{align} 
Moreover, using Lemma \ref{lemma:diag_dominance} again, we can find $\mathbf{w} \in \mathcal{C}_\varepsilon^{|\mathcal{S}(\bbeta, \mathbf{v} , \varepsilon^{1/3})^c|}$ such that 
\begin{align}
\Big\|  \sum_{i \in \mathcal{S}(\bbeta, \mathbf{v} , \varepsilon^{1/3})^c} b''(\langle \mathbf{x}_i, \bbeta \rangle ) \mathbf{x}_i \mathbf{x}_i^{\top} - \sum_{i \in \mathcal{S}(\bbeta, \mathbf{v} , \varepsilon^{1/3})^c} w_i   \mathbf{x}_i \mathbf{x}_i^{\top}  \Big\|_{\mathrm{op}} \leq \varepsilon. \nonumber 
\end{align} 
Combining, we have, 
\begin{align}
\| \mathbf{X}^{\top} \mathbf{D}_{\bbeta} \mathbf{X} - (\mathbf{M}_1 + \mathbf{M}_2) \|_{\mathrm{op}} \leq (\delta(\varepsilon) + \varepsilon) \| \mathbf{X}^{\top} \mathbf{X} \|_{\mathrm{op}}, \nonumber
\end{align}
where 
\begin{align}
\mathbf{M}_1 := \sum_{i \in \mathcal{S}(\bbeta, \mathbf{v} , \varepsilon^{1/3}) } b''(v_i ) \mathbf{x}_i \mathbf{x}_i^{\top} \in \mathcal{N}_1(p, \varepsilon), \quad \mathbf{M}_2 := \sum_{i \in \mathcal{S}(\bbeta, \mathbf{v} , \varepsilon^{1/3})^c } w_i \mathbf{x}_i \mathbf{x}_i^{\top} \in \mathcal{N}_1(p, \varepsilon). \nonumber 
\end{align}  

Finally, the net $\mathcal{N}(p,\varepsilon)$ satisfies $\lim_{\varepsilon \to 0} \lim_{p \to \infty} \frac{1}{p} \log |\mathcal{N}(p, \varepsilon)| = 0$. This completes the proof. 

\end{proof}

\subsection{Proof of Theorem \ref{thm:alg_glm}} 
\label{sec:proofs3} 
Let $\mu_{\bbeta}$ denote the following $\bbeta$-indexed distribution,
\begin{align}
        \frac{\mathrm{d} \mu_{\bbeta}}{\mathrm{d} \pi_p}(\sigma) \propto \exp{ \left \{ \sum_{j=1}^p \left [ \sigma_j \sum_{i=1}^n y_i x_{ij} - \sum_{i=1}^n  \left (
        b(x_i^T \bbeta_{\sigma_j,j}) - b(x_i^T \bbeta_{0,j}) \right )\right ] \right \} } = \exp \left  [ \nabla H ( \bsigma; \bbeta) \right ]. \label{eq:beta_dist} 
\end{align}
Then by the first display on Page 23 of \cite{austin2019structure}, we get,
\begin{equation}
\label{eq:optimization_indexed_by_beta}
        \log \mathcal{Z}_p - \left ( \sup_{Q \in \{\mu_{\bbeta}:\bbeta \in [-1,1]^p \}} \left ( \mathbb{E}_{\bsigma \sim Q} \left [ - H (\bsigma) \right ] -\operatorname{D}_{\text{KL}}(Q \| \pi_p )\right )  \right ) = o(p).
\end{equation}
This completes the proof of part (i). 

Next, we establish part (ii). To this end, recall the approximation $G_3(\bsigma; \bbeta)$ to $\nabla H(\bsigma; \bbeta)$ from \eqref{eq:G3_defn}. \eqref{eq:G3_defn} implies that 
\begin{align}
\sup_{\bsigma \in [-1,1]^p} \Big| G_3(\bsigma ; \bbeta) - \frac{b''(0) }{2}  \bsigma^{\top} \mathrm{diag}(\mathbf{X}^{\top} \mathbf{X}) \bsigma - \nabla H(\bsigma; \bbeta) \Big| = o(p). \nonumber 
\end{align} 

For $\bbeta \in [-1,1]^p$, define $V(\bbeta)= \mathbf{X}^{\top} \mathbf{y} - \mathbf{X}^{\top} b'(\mathbf{X}\bbeta) + b''(0) \mathrm{diag}(\mathbf{X}^{\top}\mathbf{X}) \bbeta$, such that 
\begin{align}
G_3(\bsigma; \bbeta) = - \bsigma^{\top} V(\bbeta). \nonumber 
\end{align} 

First observe that 
\begin{align}
D_{\mathrm{KL}}( \mu_{\bbeta} \| \pi_{(V(\bbeta), \mathbf{d})}) = \mathbb{E}_{\mu_{\bbeta}} \Big[ \log \frac{\mathrm{d}\mu_{\bbeta} }{\mathrm{d}\pi_{(V(\bbeta), \mathbf{d})} } \Big] = o(p). 
\end{align} 

Using Marton's transportation-cost inequality, we have, 
\begin{align}
\bar{d}_{W_1} ( \mathbf{X}, \mathbf{Y})  = o(1),  \label{eq:wasserstein} 
\end{align} 
where $\mathbf{X} \sim \mu_{\bbeta}$ and $\mathbf{Y} \sim \pi_{(V(\bbeta), \mathbf{d})}$. Formally, this implies that there exists a coupling $\Gamma$ between $\mu_{\bbeta}$ and $\pi_{(V(\bbeta), \mathbf{d})}$ such that $\frac{1}{p} \mathbb{E}_{\Gamma}[\| \mathbf{X} - \mathbf{Y} \|_1] = o(1)$.  As $\mathbf{X} , \mathbf{Y} \in [-1,1]^p$, $\frac{1}{p} \mathbb{E}_{\Gamma} [\| \mathbf{X} - \mathbf{Y} \|_2^2 ] \leq \frac{1}{p} \mathbb{E}_{\Gamma}[\| \mathbf{X} - \mathbf{Y} \|_1] = o(1)$.  In particular, this implies 
\begin{align}
| \mathbb{E}_{\mu_{\bbeta}}[H(\mathbf{X})] - \mathbb{E}_{\pi_{(V(\bbeta), \mathbf{d})}}[ H(\mathbf{Y})] | &= | \mathbb{E}_{\Gamma}[H(\mathbf{X}) - H(\mathbf{Y})] | \nonumber \\
&= | \mathbb{E}_{\Gamma}[\langle \frac{\partial}{\partial \mathbf{x}}  H(\mathbf{Z}) , (\mathbf{Y} - \mathbf{X}) \rangle] | \nonumber \\
&\leq \|  \| \frac{\partial}{\partial \mathbf{x}}  H  \|_2 \|_{\infty} \cdot \mathbb{E}_{\Gamma} \| \mathbf{Y} - \mathbf{X} \|_2  \nonumber \\
&\leq \sqrt{p} \|  \| \frac{\partial}{\partial \mathbf{x}}  H  \|_2\|_{\infty}  \sqrt{\frac{1}{p} \mathbb{E}_{\Gamma} \| \mathbf{Y} - \mathbf{X} \|_2^2 } = o(p). \nonumber 
\end{align} 
where the last step uses Jensen's inequality with the function $x \mapsto \sqrt{x}$. Note that by direct computation, 
\begin{align}
\nabla H(\bbeta) = - \mathbf{X}^{\top} \mathbf{y} + \mathbf{X}^{\top} b'(\mathbf{X} \bbeta) = - \mathbf{X}^{\top} (\mathbf{y} - b'(0) \boldsymbol{1}) + \mathbf{X}^{\top} \mathbf{D}_{t_* \bbeta} \mathbf{X} \bbeta,  \nonumber 
\end{align} 
where the last step follows by Mean Value Theorem and $\mathbf{D}_{t_* \bbeta} = \mathrm{diag}(b''(\langle \mathbf{x}_i , t_* \bbeta \rangle))$. Using our assumption that $\| \mathbf{X}^{\top} (\mathbf{y} - b'(0) \boldsymbol{1}) \|^2 = O(p)$ and $\| \mathbf{X}^{\top} \mathbf{D}_{\bbeta} \mathbf{X}\|_{\mathrm{op}} \lesssim \| \mathbf{X}^{\top} \mathbf{X} \|_{\mathrm{op}} = O(1)$, we have 
\begin{align}
\sup_{\bbeta \in [-1,1]^{p} } \| \| \nabla H(\bbeta) \|_2^2 \|_{\infty} = O(p). \label{eq:grad_sup} 
\end{align} 

On the other hand, 
\begin{align}
D_{\mathrm{KL}}(\mu_{\bbeta} \| \pi_p) &= \mathbb{E}_{\mu_{\bbeta}}\Big[ \log \frac{\mathrm{d} \mu_{\bbeta} }{\mathrm{d}\pi_p} \Big] = \mathbb{E}_{\mu_{\bbeta}}\Big[ \log \frac{\mathrm{d} \pi_{(V(\bbeta), \mathbf{d})}  }{\mathrm{d}\pi_p} \Big] + o(p). \nonumber \\
&= \mathbb{E}_{\mu_{\bbeta}} \Big[ - \bsigma^{\top} V(\bbeta) - \frac{b''(0)}{2} \sum_{j=1}^{p} d_j \sigma_j^2  \Big] + o(p). \nonumber 
\end{align} 

Finally, 
\begin{align}
D_{\mathrm{KL}}(\pi_{(V(\bbeta), \mathbf{d})} \| \pi_p) &= \mathbb{E}_{\pi_{(V(\bbeta), \mathbf{d})} } \Big[ \log \frac{\mathrm{d} \pi_{(V(\bbeta),\mathbf{d})}  }{\mathrm{d}\pi_p} \Big] \nonumber \\
&= \mathbb{E}_{\pi_{(V(\bbeta), \mathbf{d})} } \Big[ - \bsigma^{\top} V(\bbeta) - \frac{b''(0)}{2} \sum_{j=1}^{p} d_j \sigma_j^2  \Big]. \nonumber 
\end{align} 

Thus we have, 
\begin{align}
&\Big| D_{\mathrm{KL}}(\mu_{\bbeta} \| \pi_p) - D_{\mathrm{KL}}(\pi_{(V(\bbeta), \mathbf{d})} \| \pi_p)\Big| \nonumber \\
&\leq \| V(\bbeta) \|_2 \| \mathbb{E}_{\mu_{\bbeta}} [\bsigma] - \mathbb{E}_{\pi_{(V(\bbeta), \mathbf{d})}}[\bsigma] \|_2 +\frac{1}{2} |b''(0)| \max \{d_j\} \| \mathbb{E}_{\mu_{\bbeta}} [\bsigma] - \mathbb{E}_{\pi_{(V(\bbeta), \mathbf{d} )}}[\bsigma] \|_1 = o(p) \nonumber 
\end{align} 
using \eqref{eq:wasserstein} and the discussion around. Combining, we obtain, 
\begin{align}
\sup_{\bbeta \in [-1,1]^{p} } \Big| \Big( \mathbb{E}_{\mu_{\bbeta}}[H(\bsigma)] - D_{\mathrm{KL}}(\mu_{\bbeta} \| \pi_p) \Big) - \Big( \mathbb{E}_{\pi_{(V(\bbeta), \mathbf{d})}}[H(\bsigma)] - D_{\mathrm{KL}}(\pi_{(V(\bbeta), \mathbf{d})} \| \pi_p) \Big)   \Big| = o(p).  \label{eq:approximation} 
\end{align}
This completes part (b).
 
The proof of part (c) is direct from part(b), and is thus omitted.

\subsection{Proof of Theorem \ref{thm:structure}} 
\label{sec:proof4} 
We prove Theorem \ref{thm:structure} in this section.

\begin{proof}[Proof of Theorem \ref{thm:structure}] 

Recall the distribution $\mu_{\bbeta}$ from \eqref{eq:beta_dist}. 
\cite{austin2019structure} establishes that 
\begin{align}
\log \mathcal{Z}_p = \int \Big[ \mathbb{E}_{\mu_{\bbeta} } [-H(\bsigma)] - \mathrm{D}_{\mathrm{KL}}(\mu_{\bbeta} \| \pi_p)  \Big] \mathrm{d}\mu(\bbeta | \mathbf{y}) + o(p). 
\end{align} 
Now using \eqref{eq:approximation}, we have, 
\begin{align}
\log \mathcal{Z}_p = \int \Big[ \mathbb{E}_{\pi_{(V(\bbeta), \mathbf{d})} } [-H(\bsigma)] - \mathrm{D}_{\mathrm{KL}}(\pi_{(V(\bbeta),\mathbf{d})} \| \pi_p)  \Big] \mathrm{d}\mu(\bbeta | \mathbf{y}) + o(p).
\end{align} 
On the other hand, for any $\bbeta \in [-1,1]^{p}$, we have, 
\begin{align}
\log \mathcal{Z}_p  \geq \mathbb{E}_{\pi_{(V(\bbeta), \mathbf{d})} } [-H(\bsigma)] - \mathrm{D}_{\mathrm{KL}}(\pi_{(V(\bbeta),\mathbf{d})} \| \pi_p) . \nonumber 
\end{align} 
For any $\varepsilon >0$ setting 
\begin{align} 
T(\varepsilon) = \{ \bbeta \in [-1,1]^p: \mathbb{E}_{\pi_{(V(\bbeta), \mathbf{d})} } [-H(\bsigma)] - \mathrm{D}_{\mathrm{KL}}(\pi_{(V(\bbeta),\mathbf{d})} \| \pi_p) < \log \mathcal{Z}_p - p \varepsilon\}, \nonumber 
\end{align} 
we have $\mu(T(\varepsilon)|\mathbf{y}) \to 0$ as $p \to \infty$. Now, define the event $\mathcal{E}_p(\delta) = \{\bbeta \in [-1,1]^p: \| \dot{c}_{\pi}(V(\bbeta), \mathbf{d}) - \mathbf{u}^* \|^2 \leq p \delta  \}$. Using the well-separated optimizer property, there exists $\varepsilon>0$ such that  if $\bbeta \in \mathcal{E}_p(\delta)^c$, $\bbeta \in T(\varepsilon)$. Thus for any $\delta>0$, there exists $\varepsilon>0$ such that $\mathcal{E}_p(\delta)^c \subseteq T(\varepsilon)$. In turn, for any $\delta>0$, this directly implies $\mu(\mathcal{E}_p(\delta)^c | \mathbf{y}) \to 0$ as $p \to \infty$.

For $\bbeta \in [-1,1]^p$, define $V(\bbeta)= \mathbf{X}^{\top} \mathbf{y} - \mathbf{X}^{\top} b'(\mathbf{X}\bbeta) + b''(0) \mathrm{diag}(\mathbf{X}^{\top}\mathbf{X}) \bbeta$. 
 Let $\bbeta^*$ be an optimizer of the map  $\phi: [-1,1]^{p} \to \mathbb{R}$ 
\begin{align}
\phi(\bbeta) := \Big[ \mathbb{E}_{\pi_{(V(\bbeta), \mathbf{d} )}}[-H(\bsigma)] - \mathrm{D}_{\mathrm{KL}}(\pi_{(V(\bbeta), \mathbf{d})} \| \pi_p) \Big]. \nonumber 
\end{align} 

We note that 
\begin{align}
\log \mathcal{Z}_p - \phi(\bbeta^*) = o(p) \nonumber 
\end{align} 
implies that $\bbeta^* \in \mathcal{E}_p(\delta)$. 

Now, let $\{\mathcal{C}_1, \cdots, \mathcal{C}_{N(p,\delta')}\}$ be a $\delta'$-partition, as introduced in \cite{austin2019structure}. Formally, this implies for $\bbeta, \bgamma \in \mathcal{C}_j$, $\| V(\bbeta) - V(\bgamma) \|_1 \leq p \delta'$. Without loss of generality, assume that $\bbeta^* \in \mathcal{C}_1$. Fix a representative $\bpsi_1, \cdots, \bpsi_{N(p,\delta')}$ from each partition. In particular, we fix $\bbeta^* = \bpsi_1$ without loss of generality. Now, fix $M>0$, and define a relation $\stackrel{M}{\sim}$ such that for $j \geq 2$,  
\begin{align}
\mathcal{C}_j \stackrel{M}{\sim} \mathcal{C}_1 \,\,\,\, \textrm{if}\,\,\,\, 
\sum_{i: |V(\bbeta^*)_i| \leq M} (\dot{c}_{\pi}(V(\bpsi_j)_i, d_i) - u_i^*)^2 \leq 10 p \delta. \label{eq:M_relation} 
\end{align}

Using \eqref{eq:M_relation}, observe that if $\mathcal{C}_j \stackrel{M}{\nsim} \mathcal{C}_1$,  
\begin{align}
&\sum_{i=1}^{p} (\dot{c}_{\pi}(V(\bpsi_j)_i, d_i) - u_i^*)^2 \nonumber \\
&\geq \sum_{i: |V(\bbeta^*)_i| \leq M} (\dot{c}_{\pi}(V(\bpsi_j)_i, d_i) - u_i^*)^2  > 10 p \delta. \nonumber 
\end{align} 
Further, for any $\bbeta \in \mathcal{C}_j$, $\|V(\bbeta) - V(\bpsi_j)\|_1 \leq p \delta'$. Thus we have, by Taylor expansion, 
\begin{align}
| \dot{c}_{\pi}(V(\bbeta)_i, d_i) - \dot{c}_{\pi}(V(\bpsi_j)_i,d_i) | \leq \ddot{c}_{\pi}(\theta_i , d_i) |V(\bbeta)_i - V(\bpsi_j)_i| \leq |V(\bbeta)_i - V(\bpsi_j)_i|, \nonumber 
\end{align} 
where $\theta_i$ lies between $V(\bbeta)_i$ and $V(\bpsi_j)_i$. Further, we use the bound $\|\ddot{c}_{\pi}\|_{\infty} \leq 1$ in the inequality above. In particular, this implies 
\begin{align} 
&\sum_{i=1}^{p} (\dot{c}_{\pi}(V(\bbeta)_i, d_i) - \dot{c}_{\pi}(V(\bpsi_j)_i,d_i))^2 \nonumber \\
 &\leq  \sum_{i=1}^{p} |\dot{c}_{\pi}(V(\bbeta)_i, d_i) - \dot{c}_{\pi}(V(\bpsi_j)_i,d_i)| \cdot |\dot{c}_{\pi}(V(\bbeta)_i, d_i) - \dot{c}_{\pi}(V(\bpsi_j)_i,d_i)| \nonumber \\
 &\leq 2 \sum_{i=1}^{p} |\dot{c}_{\pi}(V(\bbeta)_i, d_i) - \dot{c}_{\pi}(V(\bpsi_j)_i,d_i)|  \leq 2 \|V(\bbeta) - V(\bpsi_j)\|_1 \leq 2p \delta'. \nonumber 
\end{align} 
Thus for any $\bbeta \in \mathcal{C}_j$, by triangle inequality
\begin{align} 
&\sum_{i=1}^{p} (\dot{c}_{\pi}(V(\bbeta)_i, d_i) - u_i^*)^2 \geq \frac{1}{4} (10p \delta - 2 p \delta') > p\delta \nonumber 
\end{align} 
 for $\delta'$ sufficiently small. In particular, this implies $\mathcal{C}_j \cap \mathcal{E}_p(\delta) = \emptyset$. As $\{\mathcal{C}_1, \cdots, \mathcal{C}_{N(p,\delta)}\}$ forms a partition of $[-1,1]^{p}$, we have, 
 \begin{align}
 \mathcal{E}_p(\delta) \subseteq \cup_{j: \mathcal{C}_j \stackrel{M}{\sim} \mathcal{C}_1} \mathcal{C}_j. \nonumber 
 \end{align} 
 Combining this with our earlier observation, we have, $\mu(\cup_{j: \mathcal{C}_j \stackrel{M}{\sim} \mathcal{C}_1} \mathcal{C}_j \backslash \mathcal{E}_p(\delta) ) \to 0$ as $p \to \infty$.
 
 Next, note that for any $j \geq 1$, $\bbeta \in \mathcal{C}_j$, by Lemma~\ref{lem:wasserstein},
 \begin{align}
& d_{W_1}(\pi_{(V(\bbeta),\mathbf{d})}, \pi_{(V(\mathbf{\bpsi}_j),\mathbf{d})}) \leq \frac{1}{p} \sum_{i=1}^{p} d_{W_1}(\pi_{(V(\bbeta)_i,d_i)} , \pi_{(V(\bpsi_j)_i, d_i)} ) \nonumber \\
&\leq C \frac{1}{p} \sum_{i=1}^{p} |V(\bbeta)_i - V(\bpsi_j)_i| \leq C \delta'. \nonumber  
 \end{align} 
 
 Using \cite{austin2019structure} and \eqref{eq:M_relation}, we have, 
 \begin{align}
 \sum_{j : \mathcal{C}_j \stackrel{M}{\sim} \mathcal{C}_1} \mu(\mathcal{C}_j) \int_{\mathcal{C}_j } d_{W_1}(\mu_{|\mathcal{C}_j} , \pi_{(V(\bbeta), \mathbf{d})} ) \mathrm{d}\mu_{|\mathcal{C}_j} (\bbeta)=o(1).  \nonumber 
 \end{align} 
 
 Using triangle inequality, 
 \begin{align}
 \sum_{j : \mathcal{C}_j \stackrel{M}{\sim} \mathcal{C}_1} \mu(\mathcal{C}_j) d_{W_1}(\mu_{|\mathcal{C}_j} , \pi_{(V(\bpsi_j), \mathbf{d})} )  &\leq \sum_{j : \mathcal{C}_j \stackrel{M}{\sim} \mathcal{C}_1} \mu(\mathcal{C}_j) \int_{\mathcal{C}_j } d_{W_1}(\mu_{|\mathcal{C}_j} , \pi_{(V(\bbeta), \mathbf{d})} ) \mathrm{d}\mu_{|\mathcal{C}_j} (\bbeta) + C\delta' \nonumber \\
 &\leq C \delta' + o(1). \nonumber 
 \end{align} 
 
Setting $\mathcal{S}(p,M) = \{ i : |V(\bbeta^*)_i| >M\}$, for $j \geq 2$ such that $\mathcal{C}_j \stackrel{M}{\sim} \mathcal{C}_1$, we have, 
\begin{align}
d_{W_1}(\pi_{(V(\bpsi_j), \mathbf{d})}, \pi_{(V(\bbeta^*), \mathbf{d})}) \leq \frac{1}{p} \sum_{i: |V(\bbeta^*)_i| \leq M} d_{W_1}(\pi_{(V(\bpsi_j)_i, d_i)}, \pi_{(V(\bbeta^*)_i,d_i)}) + \frac{2}{p} |\mathcal{S}(p,M)| , \label{eq:wasserstein_int1} 
\end{align} 
where we use the fact that $d_{W_1}(X,Y) \leq 2$ for any two random variables taking values in $[-1,1]$. Using Lemma \ref{lem:wasserstein}, the first term can be controlled as 
\begin{align}
&\frac{1}{p} \sum_{i: |V(\bbeta^*)_i| \leq M} d_{W_1}(\pi_{(V(\bpsi_j)_i, d_i)}, \pi_{(V(\bbeta^*)_i,d_i)}) \nonumber \\
&\leq \frac{1}{p} \sum_{i: |V(\bbeta^*)_i| \leq M, |V(\bpsi_j)_i | \leq 2M} | V(\bbeta^*)_i - V(\bpsi)_i| + \frac{2}{p} |\{i: |V(\bbeta^*)_i| \leq M, |V(\bpsi_j)_i | > 2M\}|. \nonumber   
\end{align}

\begin{align}
&\frac{1}{p} \sum_{i: |V(\bbeta^*)_i| \leq M, |V(\bpsi_j)_i | \leq 2M} | V(\bbeta^*)_i - V(\bpsi)_i|   \nonumber \\
&\leq \frac{1}{p} \sum_{i: |V(\bbeta^*)_i| \leq M, |V(\bpsi_j)_i | \leq 2M} \frac{1}{\ddot{c}_{\pi}(\theta_i,d_i)} | \dot{c}_{\pi}(V(\bbeta^*)_i, d_i) - \dot{c}_{\pi}(V(\bpsi_j)_i, d_i)| \nonumber \\
&\leq C \sqrt{\frac{1}{p} \sum_{i: |V(\bbeta^*)_i| \leq M}  (\dot{c}_{\pi}(V(\bbeta^*)_i, d_i) - \dot{c}_{\pi}(V(\bpsi_j)_i, d_i) )^2 } \leq C \sqrt{\delta} 
\end{align} 
for some universal constant $C>0$. On the other hand, if $|V(\bbeta^*)|_i \leq M$ and $|V(\bpsi_j)_i| > 2M$, there exists $\delta'_M$ such that 
\begin{align}
|\dot{c}_{\pi}(V(\bbeta)_i, d_i) - \dot{c}_{\pi}(V(\bpsi)_i,d_i) | \geq \delta'_{M}. \nonumber 
\end{align} 
Thus we have, 
\begin{align}
&(\delta'_M)^2 \frac{|\{ i : |V(\bbeta^*)_i| \leq M, |V(\bpsi_j)_i| > 2M \}|}{p} \nonumber \\
&\leq \frac{1}{p} \sum_{i: |V(\bbeta^*)_i| \leq M} (\dot{c}_{\pi}(V(\bbeta^*)_i, d_i) -  \dot{c}_{\pi}(V(\bpsi_j)_i, d_i))^2 \nonumber \\
&\leq \frac{2}{p} \sum_{i: |V(\bbeta^*)_i| \leq M} (u_i^* -  \dot{c}_{\pi}(V(\bpsi_j)_i, d_i))^2 + \nonumber \\
&+ \frac{2}{p} \sum_{i=1}^{p} (\dot{c}_{\pi}(V(\bbeta^*)_i ,d_i) - u_i^*)^2   \nonumber \\
&\leq 50 \delta. \nonumber 
\end{align} 
In summary, we obtain the upper bound, 
\begin{align}
 \frac{|\{ i : |V(\bbeta^*)_i| \leq M, |V(\bpsi_j)_i| > 2M \}|}{p} \leq 50 \frac{\delta}{(\delta'_M)^2}. \nonumber 
\end{align} 

Plugging these bounds back into \eqref{eq:wasserstein_int1} we have, 
\begin{align}
d_{W_1}(\pi_{(V(\bpsi_j), \mathbf{d})}, \pi_{(V(\bbeta^*), \mathbf{d})}) \leq C \sqrt{\delta} + 100 \frac{\delta}{(\delta'_M)^2} + \frac{2}{p} |\mathcal{S}(p,M)|. \nonumber 
\end{align} 

Using Lemma \ref{eq:rare_coordinates} 
%
given $\varepsilon>0$, we first choose $M$ large enough so that the last term is less than $\varepsilon/2$. Then we choose $\delta>0$ small enough so that the sum of the first two terms is less than $\varepsilon/2$. Thus we have, using \eqref{eq:M_relation}, 
\begin{align}
 \sum_{j : \mathcal{C}_j \stackrel{M}{\sim} \mathcal{C}_1} \mu(\mathcal{C}_j) d_{W_1}(\mu_{|\mathcal{C}_j} , \pi_{(V(\bbeta^*), \mathbf{d})} ) \leq C \delta' + \varepsilon. \nonumber  
\end{align} 
As $\sum_{j : \mathcal{C}_j \stackrel{M}{\nsim} \mathcal{C}_1} \mu(\mathcal{C}_j) = o(1)$ and the Wasserstein distance is bounded, we have, 
\begin{align}
 \sum_{j = 1}^{N(p,\delta)} \mu(\mathcal{C}_j) d_{W_1}(\mu_{|\mathcal{C}_j} , \pi_{(V(\bbeta^*), \mathbf{d})} )  \leq C \delta' + \varepsilon + o(1). \nonumber 
\end{align}  
Since there exists an efficient coupling between $\mu_{|\mathcal{C}_j}$ and $\pi_{(V(\bbeta^*), \mathbf{d})}$ for each $\mathcal{C}_j$, this leads to the Wasserstein bound $d_{W_1}(\mu, \pi_{(V(\bbeta^*), \mathbf{d})}) \leq C \delta' + \varepsilon + o(1)$. The proof is now complete by another application of Lemma \ref{eq:rare_coordinates} and triangle inequality.  
 
 \end{proof} 
 
 \begin{proof}[Proof of Corollary~\ref{cor:logistic_regression_classification_error}]
First of all, let $\Gamma^*$ be the optimal coupling between $\mu(\cdot | \mathbf{y})$ and $\prod_{j=1}^p \pi_{(\tau_j^* , d_j)}$, then for any Lipschitz-$J$ function $f_{p,j}:[-1,1] \to \mathbb{R}$ and $(\bbeta, \bgamma ) \sim \Gamma^*$, we have
	\begin{align}
	\label{eq:wasserstein_implication}
		\frac{1}{p} \left | \sum_{j=1}^p f_{p,j}(\beta_j) - \sum_{j=1}^p f_{p,j}(\gamma_j)\right | \leq \frac{J}{p} \left \| \bbeta - \bgamma \right \|_1 \overset{\Gamma^*}{\longrightarrow} 0.
	\end{align}
\begin{itemize}
	\item[(i)] Let $\tilde{f}_{p,j}(x):\mathbb{R} \to [0,1]$ be 
	\begin{align*}
		\tilde{f}_{p,j}(x) = \begin{cases}
      0 & \text{if $x \leq q_j^{\alpha / 2} - \varepsilon$ or $x \geq q_j^{1-\alpha/2} + \varepsilon$;}\\
      1 & \text{if $ q_j^{\alpha / 2} \leq x \leq q_j^{1-\alpha/2} $;}\\
      (x - q_j^{\alpha / 2} + \varepsilon) / \varepsilon & \text{if $q_j^{\alpha / 2} - \varepsilon \leq x \leq q_j^{\alpha / 2}$;}\\
      (x - q_j^{1-\alpha/2}) / \varepsilon & \text{if $q_j^{1-\alpha/2} \leq x \leq q_j^{1-\alpha/2} + \varepsilon$.}
    \end{cases}  
	\end{align*}
	Further more, let $f_{p,j}^{\varepsilon}$ be $\tilde{f}_{p,j}$ restricted to the domain of $[-1,1]$ and $f_p^{\varepsilon}(\bbeta) := \sum_{j=1}^p f^{\varepsilon}_{p,j}(\beta_j)$. Applying \eqref{eq:wasserstein_implication} to $f_p^{\varepsilon}$ gives
		\begin{align*}
			\frac{1}{p} \left | \sum_{j=1}^p f^{\varepsilon}_{p,j}(\beta_j) - \sum_{j=1}^p f^{\varepsilon}_{p,j}(\gamma_j)\right | \overset{\Gamma^*}{\longrightarrow} 0.
		\end{align*}
		On the other hand, Efron-Stein concentration inequality implies, for any $\delta > 0$,
		\begin{align}
			\pi^{\star} \left (\frac{1}{p} \left | \sum_{j=1}^p f^{\varepsilon}_{p,j}(\gamma_j) - \mathbb{E}\sum_{j=1}^p f^{\varepsilon}_{p,j}(\gamma_j) \right | > \delta \right ) \to 0,\nonumber
		\end{align}
		where $\pi^* := \prod_{j=1}^p \pi_{(\tau_j^*, d_j)}$. Noticing that $\frac{1}{p}\mathbb{E}\sum_{j=1}^p f^{\varepsilon}_{p,j}(\gamma_j) \ge 1-\alpha$ by definition of $\mathcal{I}_i^{\varepsilon}$'s, the two equations above render the desired result.
	\item[(ii)] Similarly, noting that $\mathbb{P}_{ \tilde{Y},\, \bbeta \sim \mu_{\bbeta| \mathbf{y}}, \hat{Y}(\bbeta)}(  \tilde{Y}  \ne \hat{Y}(\bbeta) ) = 2 \left |\mathbb{P} (  \tilde{Y} =1 ) - \mathbb{P} (  \hat{Y}(\bbeta) =1 ) \right |$, applying the same strategy to $f^g_p(\bbeta):= \sum_{j=1}^p f^{g}_{p,j}(\beta_j)= \sum_{j=1}^p (p \cdot \tilde{x}_j)\beta_j $ again gives the desired result.
\end{itemize}
\end{proof}
 
 \section{Auxiliary Results} 
 \label{sec:auxiliary_results}

 \begin{lemma}
 For $u \in [-1,1]$ and $d\geq 0$, let $Q_{u} = \pi_{(h(u,d),d)}$. Then we have, 
 \begin{itemize}
 \item[(i)] $\mathrm{D}_{\mathrm{KL}}(\pi_{(h(u,d),d)} \| \pi):= G(u,d) = u h(u,d) - b''(0) \frac{d}{2} \mathbb{E}_{\pi_{(h(u,d),d)}}[X^2] - c_{\pi}(h(u,d), d)$. 
 \item[(ii)] We have,
 \begin{align}
 \frac{\partial}{\partial u} G(u,d) = h(u,d) - b''(0) \frac{d}{2} \frac{\mathrm{Cov}_{\pi_{(h(u,d),d)}} (X,X^2) }{\ddot{c}_{\pi} (h(u,d),d)}. \nonumber 
 \end{align} 
 \item[(iii)] As $u \to 1$, $\frac{\partial}{\partial u} G(u,d)  \to \infty$. Similarly, as $u \to -1$, $\frac{\partial}{\partial u} G(u,d)  \to -\infty$. 
 \end{itemize} 
 \end{lemma} 
 
 \begin{proof}
 \begin{itemize}
 \item[(i)] By definition, 
 \begin{align}
 G(u,d) = \mathbb{E}_{\pi_{(h(u,d),d)} } \Big[ \log \frac{\mathrm{d} Q_{u} }{\mathrm{d} \pi} \Big] = \mathbb{E}_{\pi_{(h(u,d),d)}}[ X h(u,d) - b''(0) \frac{d}{2} X^2 - c_{\pi}(h(u,d),d)].\nonumber 
 \end{align} 
 We are done with part (i) once we note that $\mathbb{E}_{\pi_{(h(u,d),d)}} [X] = u$. 
 
 \item[(ii)] By differentiation, we obtain, 
 \begin{align}
 \frac{\partial}{ \partial u} G(u,d) = h(u,d)  - b''(0) \frac{d}{2} \frac{\partial }{\partial u} \mathbb{E}_{\pi_{(h(u,d),d)}}[X^2]. \nonumber 
 \end{align} 
 Further, we note that 
 \begin{align} 
  \frac{\partial }{\partial u} \mathbb{E}_{\pi_{(h(u,d),d)}}[X^2] = \mathrm{Cov}_{\pi_{(h(u,d),d)}} (X,X^2) \frac{\partial}{\partial u} h(u,d). \nonumber 
 \end{align} 
 Finally, note that $u = \dot{c}_{\pi}(h(u,d) , d)$. Differentiating in $u$, we obtain 
 \begin{align}
 \frac{\partial}{\partial u} h(u,d) = \frac{1}{\ddot{c}_{\pi}(h(u,d),d)}.  
  \nonumber 
 \end{align}
 The proof is complete upon plugging the derivative into the display above.  
 
 \item[(iii)] As $u \to 1$, $h(u,d) \to \infty$. For the second term, note that by Cauchy-Schwarz inequality and the observation that $\ddot{c}_{\pi}(h(u,d),d) = \mathrm{Var}_{\pi_{(h(u,d),d)}}(X)$,
 \begin{align}
	\Bigg | \frac{\mathrm{Cov}_{\pi_{(h(u,d),d)}} (X,X^2) }{\ddot{c}_{\pi} (h(u,d),d)} \Bigg | = \Bigg|  \frac{\mathrm{Cov}_{\pi_{(h(u,d),d)}}(X,X^2) }{ \mathrm{Var}_{\pi_{(h(u,d),d)}} (X) } \Bigg| \leq \sqrt{ \frac{ \mathrm{Var}_{\pi_{(h(u,d),d)}} (X^2)}{\mathrm{Var}_{\pi_{(h(u,d),d)}} (X) } }  
 \end{align} 
 The function $x \mapsto x^2$ is Lipschitz on $[-1,1]$ and thus by Efron-Stein inequality, $\mathrm{Var}_{\pi_{(h(u,d),d)}} (X^2) \leq C \cdot \mathrm{Var}_{\pi_{(h(u,d),d)}}(X)$ for some universal constant $C>0$. This completes the proof for the case $u \to 1$. The analysis for $u \to -1$ is analogous, and is thus omitted. 
 \end{itemize} 
 \end{proof} 
 
 \begin{lemma}
 \label{lem:wasserstein} 
 Let $\pi$ be a probability measure on $[-1,1]$. We have, 
 \begin{align}
 \sup_{d \geq 0} d_{W_1}(\pi_{(A_1,d)}, \pi_{(A_2,d)} ) \leq C |A_1 - A_2|,
 \end{align} 
 for some universal constant $C>0$. 
 \end{lemma} 

\begin{proof}
We use the duality formula for the 1-Wasserstein distance. We have, 
\begin{align}
d_{W_1}(\pi_{(A_1,d)} , \pi_{(A_2,d)} ) = \sup\{ |\mathbb{E}_{\pi_{(A_1,d)}}[f(X)] - \mathbb{E}_{\pi_{(A_2,d)}}[f(X)]| : f : [-1,1] \to \mathbb{R},\,\, |f(x) - f(y) | \leq |x-y| \}. \nonumber 
\end{align} 
For any 1-Lipschitz function $f$, differentiating in $A$ we have, 
\begin{align}
|\mathbb{E}_{\pi_{(A_1,d)}}[f(X)] - \mathbb{E}_{\pi_{(A_2,d)}}[f(X)]| \leq |\mathrm{cov}_{\pi_{(A,d)}}(f(X),X) | |A_1-A_2| \nonumber 
\end{align} 
for some $A$ in $[A_1, A_2]$. By Cauchy-Schwarz inequality
\begin{align}
|\mathrm{cov}_{\pi_{(A,d)}}(f(X),X) | \leq \sqrt{\mathrm{Var}_{\pi_{(A,d)}}(X) \mathrm{Var}_{\pi_{(A,d)}}(f(X)) } \leq \sqrt{\mathrm{Var}_{\pi_{(A,d)}}(f(X))}. 
\end{align} 
Finally, by the Efron-Stein inequality, 
\begin{align}
\mathrm{Var}_{\pi_{(A,d)}}(f(X)) \leq \frac{1}{2} \mathbb{E}_{\pi_{(A,d)}}(X-Y)^2 \leq 2. \nonumber 
\end{align} 
This completes the proof. 
\end{proof}

\begin{lemma}
\label{eq:rare_coordinates} 
Assume the conditions of Theorem \ref{thm:covering} are satisfied. Further, assume that there exists a well-separated optimizer $\mathbf{u}^*$. Let $\bbeta^*$ be a maximizer of the map 
\begin{align}
\phi(\bbeta) := \Big[ \mathbb{E}_{\pi_{(V(\bbeta), \mathbf{d} )}}[-H(\bsigma)] - \mathrm{D}_{\mathrm{KL}}(\pi_{(V(\bbeta), \mathbf{d})} \| \pi_p) \Big], \nonumber 
\end{align}
where $V(\bbeta)= \mathbf{X}^{\top} \mathbf{y} - \mathbf{X}^{\top} b'(\mathbf{X}\bbeta) + b''(0) \mathrm{diag}(\mathbf{X}^{\top}\mathbf{X}) \bbeta$. Then we have, 
\begin{itemize}
\item[(i)] 
Define $\mathcal{S}(p,M) = \{i : |V(\bbeta^*)_i| > M\}$.  Then 
\begin{align}
\lim_{M \to \infty} \lim_{p \to \infty} \frac{1}{p} |\mathcal{S}(p,M)| = 0. \nonumber 
\end{align}
\item[(ii)] Define $\boldsymbol{\tau}^* = h(\mathbf{u}^*, \mathbf{d})$. Then we have, 
\begin{align}
d_{W_1}( \pi_{(V(\bbeta^*), \mathbf{d})}, \pi_{(\boldsymbol{\tau}^*, \mathbf{d})}) = o(1) \nonumber 
\end{align} 
as $p \to \infty$. 
\end{itemize} 
\end{lemma}

\begin{proof}
Define the function $\mathcal{M} : [-1,1]^p \to \mathbb{R}$ such that 
\begin{align}
\mathcal{M}(\mathbf{u}) = \mathbb{E}_{ Q_{\mathbf{u}} }[- H(\bsigma)] - D_{\mathrm{KL}}(Q_{\mathbf{u}} \| \pi_p) = \mathbb{E}_{\bsigma \sim Q_{\mathbf{u}} }[- H(\bsigma)] - G(\mathbf{u}, \mathbf{d})  , \nonumber 
\end{align} 
where $Q_{\mathbf{u}} = \prod_{i=1}^{p} \pi_{(h(u_i,d_i), d_i)}$. Evaluating the gradient at $\mathbf{u} \in (-1,1)^p$ we have, 
\begin{align}
\nabla \mathcal{M}(\mathbf{u}) = \nabla  \mathbb{E}_{ Q_{\mathbf{u}} }[- H(\bsigma)] - h(\mathbf{u}, \mathbf{d}) + b''(0) \frac{\mathbf{d}}{2} \frac{\mathrm{Cov}_{\pi_{ (h(\mathbf{u},\mathbf{d}), \mathbf{d}) }} (X,X^2) }{\ddot{c}_{\pi}(h(\mathbf{u},\mathbf{d}),\mathbf{d}) }. \label{eq:M_derivative} 
\end{align} 
To differentiate the first term, note that 
\begin{align}
\frac{\partial }{\partial u_j} \mathbb{E}_{ Q_{\mathbf{u}} }[- H(\bsigma)] = \frac{\partial }{\partial h(u_j,d_j)} \mathbb{E}_{ Q_{\mathbf{u}} }[- H(\bsigma)] \frac{\partial h(u_j,d_j)}{\partial u_j} = \frac{\mathrm{Cov}_{\pi_{(h(\mathbf{u},\mathbf{d}),\mathbf{d})}} (-H(\bsigma), \sigma_j)}{\ddot{c}_{\pi}(h(u_j,d_j),d_j)}.  \nonumber 
\end{align}

We note that $H(\bsigma) = \sum_{i=1}^{n}    (- y_i  \mathbf{x}_i^{\top} \bsigma + b(\mathbf{x}_i^{\top} \bsigma))$.  For $1\leq j \leq p$, 
\begin{align}
H(\bsigma) = H(\bsigma_{0,j}) -  \sum_{i=1}^{n} y_i x_{ij} \sigma_j  + \sum_{i=1}^{n} (b(\mathbf{x}_i^{\top} \bsigma) - b(\mathbf{x}_i^{\top} \bsigma_{0,j})). \nonumber 
\end{align} 
Thus we have, 
\begin{align}
\mathrm{Cov}_{\pi_{(h(\mathbf{u},\mathbf{d}),\mathbf{d})} } (-H(\bsigma), \sigma_j) =  \ddot{c}_{\pi}(h(u_j,d_j),d_j) \sum_{i=1}^{n} y_i x_{ij} - \mathrm{Cov}_{\pi_{(h(u_j,d_j),d_j)} }(f_j(\sigma_j ; \mathbf{u} ), \sigma_j), \nonumber 
\end{align} 
where 
\begin{align}
f_j(\sigma_j ; \mathbf{u} ) = \mathbb{E}_{\pi_{(h(\mathbf{u},\mathbf{d}),\mathbf{d})} }\Big[\sum_{i=1}^{n} (b(\mathbf{x}_i^{\top} \bsigma) - b(\mathbf{x}_i^{\top} \bsigma_{0,j})) \Big| \sigma_j \Big]. \nonumber 
\end{align} 
In particular, this implies $\frac{\partial}{\partial u_j} \mathcal{M}(\mathbf{u}) \to - \infty$ as $|u| \to 1$. Specifically, this implies that $\mathcal{M}(\cdot)$ attains its maximum on  $[-1,1]^p$. Let $\mathbf{u}_p$ be a maximizer of $\mathcal{M}$. Using $\nabla \mathcal{M}(\mathbf{u}_p)=0$  we have, $\mathbf{u}_p = \dot{c}_\pi (\mathbf{v}_p, \mathbf{d})$, where $\mathbf{v}_p \in \mathbb{R}^p$ satisfies  
\begin{align}
v_{p,j} = \sum_{i=1}^{n} y_i x_{ij}  - \frac{\mathrm{Cov}_{\pi_{(h(u_{p,j},d_j),d_j)} }(f_j(\sigma_j; \mathbf{u}_p), \sigma_j)}{\ddot{c}_{\pi}(h(u_{p,j},d_j),d_j)} + b''(0) \frac{d_j}{2} \frac{\mathrm{Cov}_{\pi_{ (h(u_{p,j}, d_j), d_j) }} (\sigma_j^2,\sigma_j) }{\ddot{c}_{\pi}(h(u_{p,j}, d_j),d_j)}. \label{eq:v_defn} 
\end{align}  

For any $M>0$, 
\begin{align}
\frac{M^2}{p} | \{j: |v_{p,j}| > M \}|  \leq  \frac{1}{p} \sum_{j=1}^{p} v_{p,j}^2. \nonumber  
\end{align} 
Furthermore, 
\begin{align}
\frac{1}{p} \sum_{j=1}^{p} v_{p,j}^2 \leq \frac{4}{p} \| \mathbf{X}^{\top} \mathbf{y} \|^2 + (b''(0))^2  \frac{C}{p} \sum_{j=1}^{p} d_j^2 + \frac{C}{p}  \sum_{j=1}^{p} \Big( \frac{\mathrm{Cov}_{\pi_{(h(u_{p,j},d_j),d_j)} }(f_j(\sigma_j; \mathbf{u}_p), \sigma_j)}{\ddot{c}_{\pi}(h(u_{p,j},d_j),d_j)} \Big)^2,  \nonumber 
\end{align} 
where $C>0$ is a universal constant, and we use that $\mathrm{Cov}_{\pi_{ (h(u_{p,j}, d_j), d_j) }} (\sigma_j^2,\sigma_j)  \leq C \ddot{c}_{\pi}(h(u_{p,j}, d_j),d_j)$ for some universal constant $C>0$. The last assertion follows by Cauchy-Schwarz and then an application of Efron-Stein inequality. 

Finally, note that by Cauchy Schwarz inequality, 
\begin{align}
\frac{1}{p}  \sum_{j=1}^{p} \Big( \frac{\mathrm{Cov}_{\pi_{(h(u_{p,j},d_j),d_j)} }(f_j(\sigma_j; \mathbf{u}_p), \sigma_j)}{\ddot{c}_{\pi}(h(u_{p,j},d_j),d_j)} \Big)^2 
\leq \frac{1}{p}  \sum_{j=1}^{p} \frac{\mathrm{Var}_{\pi_{(h(u_{p,j},d_j),d_j)} }(f_j(\sigma_j; \mathbf{u}_p))}{\mathrm{Var}_{\pi_{(h(u_{p,j},d_j),d_j)} }(\sigma_j )}.  
\end{align} 
By differentiation, 
\begin{align} 
&f_j'(\sigma_j ; \mathbf{u}_p) = \mathbb{E}_{\pi_{( h(\mathbf{u}_p,\mathbf{d}) ,\mathbf{d})}}\Big[ \sum_{i=1}^{n} b'(\mathbf{x}_i^{\top} \bsigma) x_{ij}  \Big| \sigma_j \Big]\nonumber \\
&= \mathbb{E}_{\pi_{( h(\mathbf{u}_p,\mathbf{d}) ,\mathbf{d})}}\Big[ \sum_{i=1}^{n} b'(\mathbf{x}_i^{\top} \bsigma) x_{ij}  \Big] + \mathbb{E}_{\pi_{( h(\mathbf{u}_p,\mathbf{d}) ,\mathbf{d})}} \Big[ \sum_{i=1}^{n}x_{ij}  (b'(\mathbf{x}_i^{\top} \bsigma)  -  b'(\mathbf{x}_i^{\top} \bsigma_{j,\tau_j})) |\sigma_j\Big]. \nonumber \\
&= \mathbb{E}_{\pi_{( h(\mathbf{u}_p,\mathbf{d}) ,\mathbf{d})}}\Big[ \sum_{i=1}^{n} b'(\mathbf{x}_i^{\top} \bsigma) x_{ij}  \Big] +   \mathbb{E}_{\pi_{( h(\mathbf{u}_p,\mathbf{d}) ,\mathbf{d})}} \Big[ \sum_{i=1}^{n} x_{ij}^2 b''(\mathbf{x}_i^{\top} \bsigma_{j,*}) (\sigma_j - \tau_j) | \sigma_j \Big]. \nonumber 
\end{align} 

Thus 
\begin{align}
\| f'_j(\cdot ; \mathbf{u}_p) \|_{\infty} \leq \Big|\mathbb{E}_{\pi_{( h(\mathbf{u}_p,\mathbf{d}) ,\mathbf{d})}}\Big[ \sum_{i=1}^{n} b'(\mathbf{x}_i^{\top} \bsigma) x_{ij}  \Big] \Big|  + 2 \| b''\|_{\infty} d_j. \nonumber 
\end{align} 

By Efron-Stein inequality, 
\begin{align}
\mathrm{Var}_{\pi_{(h(u_{p,j},d_j),d_j)} }(f_j(\sigma_j; \mathbf{u}_p)) \leq \| f'_j(\cdot ; \mathbf{u}_p) \|_{\infty}^2 \mathrm{Var}_{\pi_{(h(u_{p,j},d_j),d_j)} }(\sigma_j )
\end{align} 

Combining, we obtain 
\begin{align}
&\frac{1}{p} \sum_{j=1}^{p}  \frac{\mathrm{Var}_{\pi_{(h(u_{p,j},d_j),d_j)} }(f_j(\sigma_j; \mathbf{u}_p))}{\mathrm{Var}_{\pi_{(h(u_{p,j},d_j),d_j)} }(\sigma_j )} \leq \frac{1}{p} \sum_{j=1}^{p} \| f'_j(\cdot ; \mathbf{u}_p) \|_{\infty}^2. \nonumber\\
 &\leq \frac{2}{p} \sum_{j=1}^{p} \Big( \mathbb{E}_{\pi_{( h(\mathbf{u}_p,\mathbf{d}) ,\mathbf{d})}}\Big[ \sum_{i=1}^{n} b'(\mathbf{x}_i^{\top} \bsigma) x_{ij}  \Big] \Big)^2 + 8 \| b''\|_{\infty}^2 \frac{1}{p} \sum_{j=1}^{p}  d_j^2 \leq C  \nonumber 
\end{align} 
for some universal constant $C>0$. Thus we obtain 
\begin{align}
 | \{j: |v_{p,j}| > M \}| \leq \frac{Cp}{M^2}. \nonumber 
\end{align}
In conclusion, this implies 
\begin{align}
\frac{1}{p} | \{ i: |u_{p,i}| > 1- \varepsilon \} | \leq \frac{C}{M(\varepsilon)^2}, \nonumber 
\end{align} 
where $M(\varepsilon) \to \infty$ as $\varepsilon \to 0$. 

Now, let $\mathbf{r}_p$ be any sequence of points in $[-1,1]^p$ such that $\frac{1}{p} \| \mathbf{u}_p - \mathbf{r}_p\|^2 \to 0$ as $p \to \infty$. Note that both $\mathbf{u}^*$ and $\dot{c}_{\pi}( V(\bbeta^*), \mathbf{d})$ satisfy this property.  Assume if possible that there exists $\delta'>0$ such that 
\begin{align}
\lim_{\varepsilon \to 0} \lim_{p \to \infty} \frac{1}{p} | \{i : |r_{p,i}| > 1- \varepsilon \}| > \delta'. 
\end{align} 
Consider the set of indices such that $|r_{p,i}| > 1 - \varepsilon/2$. There are at least $p \delta'$ many such indices. On the other hand, there are at most $Cp/ M(\varepsilon)^2$ many indices where $|u_{p,i}| > 1- \varepsilon$.  For $\varepsilon$ sufficiently small, there are $\Theta(p)$ indices where $|u_{p,i} - r_{p,i}| > \varepsilon/2$. This directly implies 
\begin{align}
\frac{1}{p} \| \mathbf{u}_p - \mathbf{r}_p\|^2 \geq \delta'' >0 \nonumber 
\end{align} 
which is a contradiction. 

In particular, this implies 
\begin{align}
d_{W_1}(\pi_{(V(\bbeta^*) , \mathbf{d})} , \pi_{(\boldsymbol{\tau}^*, \mathbf{d})}) =& \frac{1}{p} \sum_{\max\{ |u^*_i|, |\dot{c}_p(V(\bbeta^*)_i,d_i)|\} \leq 1 - \varepsilon }d_{W_1}(\pi_{(V(\bbeta^*)_i,d_i)} , \pi_{(\tau_i,d_i)})  \nonumber \\
&+ \frac{1}{p} \sum_{\max\{ |u^*_i|, |\dot{c}_p(V(\bbeta^*)_i,d_i)|\} >  1 - \varepsilon } d_{W_1}(\pi_{(V(\bbeta^*)_i,d_i)} , \pi_{(\tau_i,d_i)})  \nonumber \\
&= \frac{1}{p} \sum_{\max\{ |u^*_i|, |\dot{c}_p(V(\bbeta^*)_i,d_i)|\} \leq 1 - \varepsilon } | V(\bbeta^*)_i - \tau^*_i|  \nonumber \\
&+ \frac{2}{p}  | \{i: \max\{ |u^*_i|, |\dot{c}_p(V(\bbeta^*)_i,d_i)|\} >  1 - \varepsilon \} |. \nonumber \\
&\leq \frac{C(\varepsilon)}{p} \sum_{i=1}^{p}   ( \dot{c}_{\pi}(V(\bbeta^*)_i, d_i) - \dot{c}_{\pi}(\tau_i, d_i))^2 \nonumber \\
&+ \frac{2}{p}  | \{i: \max\{ |u^*_i|, |\dot{c}_p(V(\bbeta^*)_i,d_i)|\} >  1 - \varepsilon \} |. \nonumber 
\end{align} 

The required conclusion now follows by first setting $p \to \infty$, and then $\varepsilon \to 0$.  Finally, note that 
\begin{align}
\frac{1}{p} |\mathcal{S}(p,M)| = \frac{1}{p} |\{i: |\dot{c}_{\pi}(V(\bbeta^*)_i, d_i)| > 1 - \varepsilon(M)\}|, \nonumber  
\end{align} 
where $\varepsilon(M) \to 0$ as $M \to \infty$. The proof now follows immediately from the observation above. 

\end{proof}

\begin{lemma}
\label{lem:perturbation} 
Let $\{\mathbf{x}_i \in \mathbb{R}^{p} : 1\leq i \leq n\} $ be such that $\| \sum_{i=1}^{n} \mathbf{x}_i \mathbf{x}_i^{\top} \|_{\mathrm{op}} = O(1)$. Let $(d_i)_{i=1}^{n}$, $(\bar{d}_i)_{i=1}^{n}$ be such that $d_i, \bar{d}_i \geq 0$ for all $1\leq i \leq n$. Finally, assume that $|d_i - \bar{d}_i | \leq \varepsilon$ for all $1\leq i \leq n$. Then we have, 
\begin{align}
\Big\| \sum_{i=1}^{n} d_i \mathbf{x}_i \mathbf{x_i}^{\top} - \sum_{i=1}^{n} \bar{d}_i \mathbf{x}_i \mathbf{x}_i^{\top}  \Big\|_{\mathrm{op}} \leq C \varepsilon \nonumber 
\end{align} 
for some universal constant $C>0$.  
\end{lemma}

\begin{proof}
Define $\mathbf{D} = \mathrm{diag}(d_1, \ldots, d_n)$ and $\bar{\mathbf{D}} = \mathrm{diag}(\bar{d}_1, \ldots, \bar{d}_n)$. We have $\sum_{i=1}^{n} d_i \mathbf{x}_i \mathbf{x}_i^{\top} = \mathbf{X}^{\top} \mathbf{D} \mathbf{X}$ and $\sum_{i=1}^{n} \bar{d}_i \mathbf{x}_i \mathbf{x}_i^{\top} = \mathbf{X}^{\top} \bar{\mathbf{D}} \mathbf{X}$. In turn, we have, 
\begin{align}
\| \mathbf{X}^{\top} \mathbf{D} \mathbf{X} - \mathbf{X}^{\top} \bar{\mathbf{D}} \mathbf{X} \|_{\mathrm{op}} = \max\{ |\lambda_{\max}(\mathbf{X}^{\top} \mathbf{D} \mathbf{X} - \mathbf{X}^{\top} \bar{\mathbf{D}} \mathbf{X})|, |\lambda_{\min}(\mathbf{X}^{\top} \mathbf{D} \mathbf{X} - \mathbf{X}^{\top} \bar{\mathbf{D}} \mathbf{X})|  \}. \nonumber 
\end{align}
For any $\mathbf{v} \in \mathbb{R}^n$ with $\| \mathbf{v} \|_2 = 1$, we have, 
\begin{align}
\Big| \mathbf{v}^{\top} \mathbf{X}^{\top} (\mathbf{D} - \bar{\mathbf{D}}) \mathbf{X} \mathbf{v} \Big| \leq \sum_{i=1}^{n} |d_i - \bar{d}_i| \langle \mathbf{x}_i, \mathbf{v} \rangle^2 \leq \varepsilon \Big\|  \sum_{i=1}^{n} \mathbf{x}_i \mathbf{x}_i^{\top}\Big\|_{\mathrm{op}}. \nonumber  
\end{align}
The proof is complete once we note that  
\begin{align}
\max\{ |\lambda_{\max}(\mathbf{X}^{\top} \mathbf{D} \mathbf{X} - \mathbf{X}^{\top} \bar{\mathbf{D}} \mathbf{X})|, |\lambda_{\min}(\mathbf{X}^{\top} \mathbf{D} \mathbf{X} - \mathbf{X}^{\top} \bar{\mathbf{D}} \mathbf{X})|\}  \leq \sup_{\|\mathbf{v}\|_2 =1} \Big|  \mathbf{v}^{\top} \mathbf{X}^{\top} (\mathbf{D} - \bar{\mathbf{D}}) \mathbf{X} \mathbf{v} \Big|. \nonumber 
\end{align}  
\end{proof}

\begin{lemma}
\label{lemma:diag_dominance} 
Let $\mathbf{M}_1$, $\mathbf{M}_2$ be $n \times n$ real symmetric matrices. For any symmetric matrix $\mathbf{M}$, define $\mathrm{diag}(\mathbf{M})$ as a diagonal matrix with $\mathrm{diag}(\mathbf{M})_{ii} = \mathbf{M}_{ii}$. Then we have, 
\begin{align}
\| \mathrm{diag}(\mathbf{M}_1) - \mathrm{diag}(\mathbf{M}_2) \|_{\mathrm{op}} \leq \| \mathbf{M}_1 - \mathbf{M}_2 \|_{\mathrm{op}}. \nonumber 
\end{align} 
\end{lemma} 

\begin{proof} 
As $\mathrm{diag}(\mathbf{M}_1)$ and $\mathrm{diag}(\mathbf{M}_2)$ are diagonal matrices, we have, 
\begin{align}
\| \mathrm{diag}(\mathbf{M}_1) - \mathrm{diag}(\mathbf{M}_2) \|_{\mathrm{op}} \leq \max_{1\leq i \leq n} |(\mathbf{M}_1)_{ii} - (\mathbf{M}_2)_{ii}| \leq \| \mathbf{M}_1 - \mathbf{M}_2 \|_{\mathrm{op}}. \nonumber 
\end{align}
This completes the proof.  

\end{proof}

\bibliographystyle{plain}
\bibliography{ref.bib}

%
%
%
%
%
%
%
%
%
%
%

\end{document}